\documentclass[twoside,12pt]{amsart}
\usepackage{amssymb}
\usepackage{amsthm}
\usepackage{amsfonts}
\usepackage{amsmath}
\usepackage[T1]{fontenc}
\usepackage{lscape}
\usepackage{url}
\usepackage{longtable}
\usepackage{graphicx}
\usepackage{latexsym}
\usepackage{colortbl}
\usepackage[all]{xy}

\newcommand{\C}{{\mathbb C}}       

\newcommand{\Aut}{\mbox{\rm Aut}}
\newcommand{\Inn}{\mbox{\rm Inn}}
\newcommand{\Out}{\mbox{\rm Out}}
\newcommand{\Hom}{\mbox{\rm Hom}}
\newcommand{\Ker}{\mbox{\rm Ker}}
\newcommand{\Ind}{\mbox{\rm Ind}}

\newcommand{\Res}{\mbox{\rm Res}}
\newcommand{\Irr}{\text{\rm Irr}}

\newcommand{\diag}{\text{\rm diag}}
\newcommand{\Inndiag}{\text{\rm Inndiag}}
\newcommand{\GL}{\text{\rm GL}}
\newcommand{\GU}{\text{\rm GU}}
\newcommand{\SL}{\text{\rm SL}}
\newcommand{\PGL}{\text{\rm PGL}}
\newcommand{\PGU}{\text{\rm PGU}}
\newcommand{\PSL}{\text{\rm PSL}}
\newcommand{\SU}{\text{\rm SU}}
\newcommand{\PSU}{\text{\rm PSU}}
\newcommand{\SO}{\text{\rm SO}}

\newcommand{\PSp}{\text{\rm PSp}}
\newcommand{\Aff}{\text{\rm Aff}}
\newcommand{\St}{\text{\rm St}}
\renewcommand{\P}{\text{\rm P}}

\newcommand{\id}{\mbox{\rm id}}

\newcommand{\ad}{\mbox{\rm ad}}

\newtheorem{num}{Notation}[subsection]
\newtheorem{defn}[num]{Definition}

\newtheorem{thm}[num]{Theorem}
\newtheorem{lem}[num]{Lemma}
\newtheorem{prp}[num]{Proposition}
\newtheorem{cor}[num]{Corollary}
\newtheorem{conj}[num]{Conjecture}

\newtheorem{xmpls}[num]{Examples}
\newtheorem{rem}[num]{Remark}

\newtheorem{notn}[num]{Notation}

\numberwithin{equation}{section}
\numberwithin{table}{section}
\numberwithin{figure}{section}
\begin{document}

\title{The eigenvalue one property of finite groups, I}
\author{Gerhard Hiss}
\author{Rafa{\l} Lutowski}

\address{Gerhard Hiss:
Lehrstuhl f{\"u}r Algebra und Zahlentheorie,
RWTH Aachen University,
52056 Aachen, Germany}
\email{gerhard.hiss@math.rwth-aachen.de}

\address{Rafa{\l} Lutowski:
Institute of Mathematics, Physics and Informatics,
University of Gda\'nsk,
ul. Wita Stwosza 57,
80-308 Gda\'nsk, Poland}
\email{rafal.lutowski@ug.edu.pl}



\begin{abstract}
We prove a conjecture of Dekimpe, De Rock and Penninckx concerning the 
existence of eigenvalues one in certain elements of finite groups acting 
irreducibly on a real vector space of odd dimension. This yields a sufficient 
condition for a closed flat manifold to be an $R_{\infty}$-manifold.
\end{abstract}

\date{\today}

\maketitle

\section{Introduction}

The purpose of this series of two articles is to prove a conjecture of Dekimpe, 
De Rock and Penninckx \cite[Conjecture~$4.7$]{DDRP}. Our results have been 
announced in~\cite{HLIschia}.

\subsection{Motivation, the conjecture, and the main result}

Let~$M$ be a real closed manifold with fundamental group~$\pi_1(M)$ and let
$f\colon M \rightarrow M$ be a homeomorphism of~$M$. The Reidemeister number~$R(f)$
of~$f$ is the number of $f_\#$-conjugacy classes on $\pi_1(M)$, where~$f_\#$ is
the automorphism of $\pi_1(M)$ induced by functoriality. A priori,~$R(f)$ is a 
positive integer or~$\infty$. If~$M$ is a nil-manifold and $R(f) = \infty$, then 
$L(f) = N(f) = 0$, where~$L(f)$ and $N(f)$ denote the Lefshetz number and the 
Nielsen number of~$f$, respectively; see the introduction of~\cite{DDRP}. If 
$R(f) = \infty$ for every homeomorphism~$f$ of~$M$, then~$M$ is called an 
$R_\infty$-manifold.

A closed flat manifold~$M$ is a space of the form 
$M = \Gamma \backslash \mathbb{R}^n$, where~$\Gamma$ is a discrete, torsion 
free, cocompact subgroup of $\Aff( \mathbb{R}^n ) =
\mathbb{R}^n \rtimes \GL_n( \mathbb{R} )$. Then~$\Gamma$ is the fundamental
group of~$M$ and there is a short exact sequence
\begin{equation}
\label{GammaSequence}
0 \longrightarrow \mathbb{Z}^n \longrightarrow \Gamma \longrightarrow G \longrightarrow 1,
\end{equation}
where $\mathbb{Z}^n = \Gamma \cap \mathbb{R}^n$ and~$G$ is a finite group, the
holonomy group of~$M$. The sequence~(\ref{GammaSequence}) gives rise to a 
representation
\begin{equation}
\label{Rho}
\gamma \colon G \rightarrow \GL_n( \mathbb{Z} ),
\end{equation}
the holonomy representation of~$M$.

In \cite[Theorem~$4.7$]{DDRP}, the authors gave a necessary condition for a
closed flat manifold~$M$ to be an $R_{\infty}$-manifold in terms of its
holonomy representation. In order to rephrase this condition, we need to 
introduce further terminology. A $\mathbb{Z}$-subresentation of~$\gamma$ is a 
representation $\rho\colon G \rightarrow \GL_d( \mathbb{Z} )$ arising from a 
$\gamma(G)$-invariant, pure sublattice~$Y$ of $\mathbb{Z}^n$ of rank~$d$; 
here,~$Y$ is called pure, if some $\mathbb{Z}$-basis of~$Y$ extends to a 
$\mathbb{Z}$-basis of~$\mathbb{Z}^n$. By concatenating~$\gamma$ with the natural
embeddings $\GL_d( \mathbb{Z} ) \rightarrow \GL_d( \mathbb{Q} )
\rightarrow \GL_d( \mathbb{R} )$, we may optionally view~$\gamma$ as a 
$\mathbb{Q}$-representation or an $\mathbb{R}$-representation of~$G$, for which
we write~$\gamma^{\mathbb{Q}}$ or~$\gamma^{\mathbb{R}}$, respectively. An 
analogous convention is used for subrepresentations of~$\gamma$. We can now 
state the criterion of Dekimpe, De Rock and Penninckx, adopting the notation 
used later on in this article.

\begin{thm}{\cite[Theorem~$4.7$]{DDRP}}
\label{DDRPTheorem}
Let~$M$ be a closed flat manifold with holonomy 
representation~{\rm (\ref{Rho})}. Suppose there is a
$\mathbb{Z}$-sub\-re\-pre\-sent\-ati\-on 
$\rho \colon G \rightarrow \GL_d( \mathbb{Z} )$, such that~$\rho^{\mathbb{Q}}$ is 
irreducible and of multiplicity one as a composition factor 
of~$\gamma^{\mathbb{Q}}$, and such that the following two conditions are 
satisfied:
%
%

{\rm (a)} If~$\tilde{\rho}$ is a $\mathbb{Q}$-subrepresentation of~$\gamma$ of 
degree~$d$ such that $\rho(G)$ and $\tilde{\rho}(G)$ are conjugate in 
$\GL_d( \mathbb{Q} )$, then~$\rho^{\mathbb{Q}}$ and~$\tilde{\rho}^{\mathbb{Q}}$ 
are equivalent. 

{\rm (b)} For every $n \in N_{\GL_d( \mathbb{Z} )}(\rho(G) )$, there is 
$g \in G$ such that $\rho(g)n$ has eigenvalue~$1$.

Then $M$ is an $R_\infty$-manifold. \hfill{\textsquare} 
\end{thm}

Subsequent to this theorem, the authors formulate the following conjecture,
which we cite, up to the notation, literally.

\begin{conj}{\cite[Conjecture~$4.8$]{DDRP}}
\label{DDRPConjecture}
Let $\rho\colon G \rightarrow \GL_d( \mathbb{Z} )$ be a representation of a 
non-trivial finite group~$G$ such that $\rho^{\mathbb{R}}$ is faithful 
and irreducible. Suppose that~$d$ is odd. Then for every 
$n \in N_{\GL_d( \mathbb{Z} )}(\rho(G) )$, there is $g \in G$ such that 
$\rho(g)n$ has eigenvalue~$1$. \hfill{\textsquare} 
\end{conj}

The authors give an example which shows that the condition on~$d$ to be odd is 
necessary. 

Condition~(b) in Theorem~\ref{DDRPTheorem} implies that~$\rho^{\mathbb{R}}$ is 
irreducible and that $N_{\GL_d( \mathbb{Z} )}(\rho(G) )$ has finite order; see 
the proof of~\cite[Theorem~A]{SzczOut}. 
In this paper we prove a slightly stronger version of 
Conjecture~\ref{DDRPConjecture}. Namely, we start with an irreducible 
$\mathbb{R}$-representation~$\rho$ of a finite group~$G$ of odd degree, but not
necessarily realizable over the integers. Then an element in 
$N_{\GL_d( \mathbb{R} )}(\rho(G) )$ need not be of finite order. We also 
do not insist on~$\rho$ being faithful. This relaxation is useful for inductive
purposes, but does not provide a true generalization, as the eigenvalue 
condition only concerns the image~$\rho(G)$. 

\begin{defn}
\label{RLDefinition}
{\rm
Let~$G$ be a finite group and let~$V$ be an $\mathbb{R}G$-module
affording the representation $\rho\colon G \rightarrow \GL(V)$. Let
$n \in N_{\GL(V)}( \rho(G) )$ be of finite order. 

We say that $(G,V,n)$ has the \emph{eigenvalue one property}, if there is
$g \in G$ such that $\rho(g)n$ has eigenvalue~$1$.

We say that~$(G,V)$ has the \emph{eigenvalue one property} if $(G,V,n')$ has the
eigenvalue one property for all $n' \in N_{\GL(V)}( \rho(G) )$ of finite order.

We say that~$G$ has the \emph{eigenvalue one property} if $(G,V')$ has the 
eigenvalue one property for every irreducible, non-trivial 
$\mathbb{R}G$-module~$V'$ of odd dimension.
} \hfill{\textsquare}
\end{defn}

\begin{xmpls}
\label{TwoExamples}
{\rm
(a) If~$V$ is the trivial $\mathbb{R}G$-module, then $(G,V)$ does not have the
eigenvalue one property.

(b) An elementary abelian $p$-group has the eigenvalue one property.
}\hfill{\textsquare}
\end{xmpls}

We can now present the main result of our series of two papers. 

\begin{thm}
\label{RLTheorem}
Every finite group has the eigenvalue one property. \hfill{\textsquare}
\end{thm}
Clearly, Theorem~\ref{RLTheorem} implies Conjecture~\ref{DDRPConjecture}.
The following is a consequence of this and Theorem~\ref{DDRPTheorem} .

\begin{cor}
\label{cor:1}
Let~$M$ be a closed flat manifold with holonomy group $G$ and holonomy
representation $\gamma\colon G \rightarrow \GL_n( \mathbb{Z} )$. 

Suppose there is a $\mathbb{Z}$-subrepresentation 
$\rho\colon G \rightarrow \GL_d( \mathbb{Z} )$ of~$\gamma$ such that 
$\rho^{\mathbb{R}}$ is 
irreducible, non-trivial, of odd degree and of multiplicity one as a
composition factor of $\gamma^{\mathbb{R}}$, and such that $\rho^{\mathbb{Q}}$
satisfies condition~{\rm (a)} of {\rm Theorem~\ref{DDRPTheorem}}.

Then~$M$ is an $R_{\infty}$-manifold. \hfill{\textsquare}
\end{cor}

Corollary~\ref{cor:1} for solvable groups~$G$ has been proved by Lutowski and
Szczepa{\'n}ski in \cite[Theorem~$1.4$]{LuSzcz}.

\subsection{Two methods}
\label{TwoMethods}
Let $(G,V,n)$ and~$\rho$ be as in Definition~\ref{RLDefinition}. Notice that
$(G,V,-n)$ is a triple with the same properties. If $g' \in G$ is such that
$\rho(g')(-n)$ has eigenvalue~$1$, then $\rho(g')n$ has eigenvalue~$-1$. 
So if $(G,V)$ has the eigenvalue one property, there are $g, g' \in G$
such that~$\rho(g)n$ and~$\rho(g')n$ have eigenvalues~$1$ and~$-1$,
respectively.

In the course of our work, we have developed several methods to prove the 
eigenvalue one property for $(G,V,n)$. Let us present the two most important
ones. The first we call the \emph{restriction method}. Suppose there is 
$H \leq G$ and an $\mathbb{R}H$-submodule $V_1 \leq V$ such that~$n$ 
normalizes~$\rho(H)$, and that~$V_1$ is~$n$ invariant. Then $(G,V,n)$ has the 
eigenvalue one property if~$(H,V_1)$ has. Since $\dim(V)$ is odd, the 
restriction of~$V$ to~$H$ contains a homogeneous component~$V_1$ of odd 
dimension. In order to proceed this way, we prove that $(H,V_1)$ has the 
eigenvalue one property, if~$(H,S)$ has, where~$S$ is an irreducible 
$\mathbb{R}H$-submodule of~$V_1$. In an inductive situation we may assume that 
the latter holds, provided~$S$ is non-trivial.  The issue with this method is to 
find a suitable subgroup~$H$ for which, in particular,~$S$ can be chosen to be 
non-trivial.

The second method is the \emph{large degree method}. Roughly speaking, if 
$\dim(V)$ is larger than a bound depending on group theoretical invariants 
derived from~$n$, then $(G,V,n)$ has the eigenvalue one property. Let us be more 
precise. If~$|n|$ is odd, then~$n$ has eigenvalue~$1$, as $\dim(V)$ is odd. 
Suppose then that~$|n|$ is even and let $M_G(n)$ denote the maximum of the 
numbers $|C_{\rho(G)}( n' )|$, where~$n'$ runs through the non-trivial powers 
of~$n$.  If $\dim(V) > (|n|-1)M_G(n)^{1/2}$, then $(G,V,n)$ has the eigenvalue 
one property. Of course, this condition might not be satisfied right away, but 
can be achieved in many cases by replacing~$n$ with $\rho(g)n$ for a suitable 
$g \in G$. Whenever we have to choose such a~$g$ explicitly, we choose~$g$ as
an involution with favorable properties. This method requires a thorough 
knowledge about the automorphisms of~$G$ and their fixed point subgroups.

In many instances, more than one of these methods could be applied. At any stage 
of our work, we have tried to apply the most elementary methods able to deal 
with the case under consideration.

\subsection{Survey of the paper}
Let us now give a survey on the structure of this article and the contents
of the individual sections. Our strategy is by contraposition, i.e.\ we assume
that Theorem~\ref{RLTheorem} is false. By a minimal counterexample we mean a 
finite group of minimal order which does not have the eigenvalue one property. 
Section~\ref{NotationAndPreliminaries} is devoted to notation and preparatory 
material. In Section~\ref{EigenvalueOne}, following a series of reductions, we 
prove that a minimal counterexample is a non-abelian finite simple group. This 
shows in particular that a solvable group has the eigenvalue one property. The 
remaining sections and Part II of our series are devoted to the proof that no
non-abelian finite simple group is a minimal counterexample, thereby proving
Theorem~\ref{RLTheorem}. 

In Section~\ref{SporadicAndAlternating} we develop some criteria which guarantee
that a triple $(G,V,n)$ has the eigenvalue one property. These conditions are
enough to rule out the sporadic simple groups, the Tits simple group and the 
alternating groups (with one exception) as minimal counterexamples; see
Corollary~\ref{AlternatingAndSporadicGroups}. Our argument is based on the fact 
that the automorphism group of any such group is a split extension of the inner 
automorphism group with a group of order at most two. 

It thus remains to consider the finite simple groups of Lie type. We introduce 
these groups and recall some of their properties in 
Section~\ref{SimpleGroupsOfLieType}. Here, we largely follow the 
book~\cite{GLS}. Of particular relevance is the description of the 
automorphisms of these groups, for which we use \cite[Section~$2.5$]{GLS}.

Section~\ref{OddCharacteristic} is devoted to the simple groups~$G$ of Lie type 
of odd characteristic. The real irreducible characters of~$G$ of odd degrees are 
easily classified with Harish-Chandra theory. Such characters are very rare. 
For example, if $G = E_6(q)$, the simple Chevalley group of type~$E_6$ with~$q$ 
odd, then~$G$ has exactly~$8$ real, 
irreducible characters of odd degree, independently of~$q$. The relevant 
information on Harish-Chandra induced characters is obtained by computations 
inside the Weyl group of~$G$. These computations are performed with the Chevie 
system based on GAP3; see~\cite{chevie}, \cite{Michel} and \cite{GAP3}. 
It should be noted, however, that some of the groups of very small Lie rank
require considerable work and lengthy calculations, due to their rather 
restricted subgroup structure. Nevertheless, our arguments are rather 
elementary, only using information about the automorphism group of~$G$ and 
simple facts from Harish-Chandra theory. The main result of this section is
Theorem~\ref{MainOdd}. 

By far the most difficult cases are provided by the finite groups~$G$ of Lie 
type of even characteristic. These will be handled in Part~II of our series.

\section{Notation and preliminaries}
\label{NotationAndPreliminaries}

Here, we introduce our notation, which is mostly standard, and collect 
miscellaneous preliminary results for later reference.

\subsection{Numbers} Let~$m$ be a non-zero integer and~$p$ a prime. We then 
write $m_p$ and $m_{p'}$ for the $p$-part, respectively the $p'$-part of~$m$.

If $x \in \mathbb{R}$, we write $\lfloor x \rfloor$ for the greatest integer 
smaller than or equal to~$x$ and $\lceil x \rceil$ for the smallest integer 
greater than or equal to~$x$.

\subsection{Graphs}
The \textit{valency} of a node of an undirected graph without loops is the 
number of edges emanating from the node. The valency of the node of a Dynkin 
diagram is the valency in the underlying undirected graph. A node of 
valency~$1$ of a tree is called a \textit{leaf}.

\subsection{Groups}
\label{Groups}
Let~$G$ be a group. For $g, x \in G$, we put ${^x\!g} := x gx^{-1}$.
Also, $\ad_x$ denotes the inner automorphism of~$G$ corresponding to~$x$, i.e.\
$\ad_x \colon G \rightarrow G$, $g \mapsto {^x\!g}$. An element $g \in G$ is called
\textit{real in~$G$}, if~$g$ is conjugate to its inverse in~$G$. If the 
surrounding group is clear form the context, we just say that~$g$ is 
\textit{real}. If $X \subseteq G$ is a subset, 
we write $\langle X \rangle$ for the subgroup of~$G$ generated by~$X$. The 
automorphism group of~$G$ is denoted by $\Aut(G)$, the group of inner 
automorphisms by $\Inn(G)$, and $\Out(G) := \Aut(G)/\Inn(G)$ is the group of 
outer automorphisms of~$G$.

If~$G$ is finite and~$p$ is a prime, we write $O^{p'}(G)$ for the smallest 
normal subgroup of~$G$ of $p'$-index.

\subsection{Matrices} Let $\Theta$ be a commutative ring and $n$ a positive 
integer. By $\diag( \zeta_1, \ldots , \zeta_n )$ we denote the diagonal
$(n \times n)$-matrix over~$\Theta$ with entries $\zeta_i \in \Theta$ at
position~$(i,i)$ for $1 \leq i \leq n$, and all other entries equal to~$0$.
This notation is extended to block diagonal matrices, where 
$(1 \times 1)$-blocks are identified with their entries. Thus, e.g.,
if $A_i$ is an $(n_i \times n_i)$-matrix for $1 = 1, 2$ and 
$\zeta, \xi \in \Theta$, then $\diag(A_1,\zeta,\xi,A_2)$ denotes the
square block diagonal matrix
$$\left( \begin{array}{c|c|c|c} A_1 & 0 & 0 & 0 \\ \hline
0 & \zeta & 0 & 0 \\ \hline
0 & 0 & \xi & 0 \\ \hline
0 & 0 & 0 & A_2 \end{array}
\right)
$$
with $n_1+n_2+2$ rows, where the $0$'s indicate matrices of zeroes of the
appropriate sizes.

The superscript~$t$ on a matrix indicates its transpose.

\subsection{Characters and modules}
Let~$G$ be a finite group and let~$K$ be a field. The $KG$-modules considered 
will always be left $KG$-modules and finite dimensional. By $\Irr(G)$ we denote 
the set of characters of the irreducible $\mathbb{C}G$-modules. A complex 
character of~$G$ is called \textit{linear}, if it has degree~$1$. Linear 
characters are irreducible. The trivial character of a subgroup $H \leq G$ is 
denoted by~$1_H$. For $\chi \in \Irr(G)$, we write $\nu_2(\chi)$ for
the Frobenius-Schur indicator of~$\chi$; see \cite[p.~$49$,~$50$]{Isaacs}. The
usual inner product of two complex valued class functions~$\chi$ and~$\psi$
of~$G$ is denoted by $\langle \chi, \psi \rangle$. From the context, there 
should not be any confusion with our notation for subgroup generation.
Suppose that $H \leq G$ and that~$\psi$ and~$\chi$ are $K$-valued class
functions of ~$H$, respectively~$G$. Then $\Ind_H^G( \psi )$ and
$\Res^G_H( \chi )$ denote the induced and restricted class functions of~$G$,
respectively~$H$. 

We collect a few of well known results on real representations of odd degree.

\begin{lem}
\label{AbsolutelyIrreducible}
Let~$V$ be an irreducible $\mathbb{R}G$-module of odd dimension. Then~$V$
is absolutely irreducible.

Also, if $\chi \in \Irr(G)$ is real with $\chi(1)$ odd, then $\nu_2(\chi) = 1$,
i.e.\ $\chi$ is realizable over~$\mathbb{R}$.
\end{lem}
\begin{proof}
If~$V$ is not absolutely irreducible, then $\mathbb{C} \otimes_{\mathbb{R}} V$
is a direct sum of two $\mathbb{C}G$-modules of equal dimension, contradicting
the odd-dimensionality of~$V$.

As $\chi$ is real, $\nu_2(\chi) = 1$ or $\nu_2(\chi) = -1$; see 
\cite[Theorem~$4.5$]{Isaacs}. As $\chi(1)$ is odd, $\nu_2(\chi) = 1$ and $\chi$ 
is afforded by a real representation; see \cite [p.~$58$]{Isaacs}.
\end{proof}

Let~$V$ be an irreducible $\mathbb{R}G$-module of odd dimension. In view of the 
above lemma, the character~$\chi$ of~$V$ equals the character 
of~$\mathbb{C} \otimes_{\mathbb{R}} V$, and we view~$\chi$ as an element 
of~$\Irr(G)$.

\begin{lem}
\label{OddDegreeClifford}
Let $V$ be an irreducible $\mathbb{R}G$-module of odd dimension and let 
$H \unlhd G$. Suppose $$\Res^G_H( V ) = V_1 \oplus \cdots \oplus V_r,$$
where the $V_i$ are the homogeneous components of $\Res^G_H( V )$.
Then each $V_i$ has odd dimension. If $|H|$ is odd, $H$ acts trivially 
on~$V$.
\end{lem}
\begin{proof}
As the $V_i$ are conjugate by the action of~$G$ (see 
\cite[Theorem~$6.5$]{Isaacs}), they all have the same dimension.

Now suppose that $|H|$ is odd, and let $S$ be an irreducible constituent 
of~$V_1$. Then~$S$ has odd dimension and thus is an absolutely irreducible
$\mathbb{R}H$-module. It follows that~$S$ is the trivial module, and thus
$V_1$ is a direct sum of trivial modules and $r = 1$. This yields our 
second assertion.
\end{proof}

\begin{lem}
\label{RestrictionReduction}
Let $H \leq G$. Let~$V$ be a non-trivial irreducible $\mathbb{R}G$-module of
odd dimension. Then there is an irreducible $\mathbb{R}H$-module~$S$ of odd
dimension which occurs with odd multiplicity in $\Res^G_H( V )$. If $C \unlhd H$
has odd order, then~$C$ acts trivially on~$S$.
\end{lem}
\begin{proof}
Write $\Res^G_H( V ) = V_1 \oplus \cdots \oplus V_r$ for the decomposition of
$\Res^G_H( V )$ into homogeneous components. Then~$V_i$ has odd dimension for
some $1 \leq i \leq r$, and an irreducible constituent~$S$ of~$V_i$ satisfies 
the requirements. We are done by Lemma~\ref{OddDegreeClifford}.
\end{proof}

\begin{lem}
\label{CharacterEstimate}
Let~$G'$ be a finite group and let $G \unlhd G'$ such that $G'/G$ is abelian. 
Let $\chi' \in \Irr(G')$ such that $\Res^{G'}_G( \chi' )$ is irreducible.
Then 
$$|\chi'(x)| \leq |C_G( x )|^{1/2}$$
for all $x \in G'$.
\end{lem}
\begin{proof}
Let $x\in G'$.
We have $C_{G'}(x)/C_G(x) = C_{G'}(x)/(G \cap C_{G'}(x)) \cong GC_{G'}(x)/G 
\leq G'/G$. Thus $|C_{G'}(x)| \leq |G'/G||C_G(x)|$.

As $\Res^{G'}_G( \chi' )$ is irreducible, $\beta\chi' \in \Irr(G')$ for
every $\beta \in \Irr(G'/G)$, and $\beta\chi' \neq \beta'\chi'$ for 
$\beta, \beta' \in \Irr(G'/G)$ with $\beta \neq \beta'$. Now $|\beta\chi'(x)|
= |\chi'(x)|$ for every $\beta \in \Irr(G'/G)$. By the 
second orthogonality relation we obtain $|G'/G||\chi'(x)|^2 \leq |C_{G'}(x)|$.
This yields our claim.
\end{proof}

\begin{lem}
\label{AllConstituents}
Let the notation and hypothesis be as in {\rm Lemma~\ref{CharacterEstimate}}.
Let $x \in G'$ and put 
$M := \max\{ |C_G(y)| \mid y \in \langle x \rangle \setminus \{ 1 \} \}$.
Suppose that $\chi'(1) > (|x| - 1)M^{1/2}$. Then 
$\Res^{G'}_{\langle x \rangle}( \chi' )$ contains every irreducible character 
of~$\langle x \rangle$ as a constituent.
\end{lem}
\begin{proof}
Let $\lambda \in \Irr( \langle x \rangle )$. Then 
$$|x|\langle \Res^{G'}_{\langle x \rangle}( \chi' ), \lambda \rangle = \chi'(1) + 
\sum_{1 \neq y \in \langle x \rangle} \chi'(y)\lambda(y^{-1}).$$
Moreover, 
\begin{eqnarray*}
\left|\sum_{1 \neq y \in \langle x \rangle} \chi'(y)\lambda(y^{-1})\right| & \leq &
\sum_{1 \neq y \in \langle x \rangle} |\chi'(y)\lambda(y^{-1})| \\
 & = & \sum_{1 \neq y \in \langle x \rangle} |\chi'(y)| \\
& \leq & \sum_{1 \neq y \in \langle x \rangle} |C_G(y)|^{1/2} \\
& \leq & (|x|-1)M^{1/2},
\end{eqnarray*}
where the penultimate estimate arises from Lemma~\ref{CharacterEstimate}.
This proves our assertion.
\end{proof}

\section{The reduction to finite simple groups}
\label{EigenvalueOne}

In this section~$G$ is a finite group. Tensor products of $\mathbb{R}$-vector
spaces are tensor products over~$\mathbb{R}$, and we usually write~$\otimes$
instead of~$\otimes_\mathbb{R}$. The phrase ``eigenvalue one property'' in its 
three specifications introduced in Definition~\ref{RLDefinition}, will 
henceforth be abbreviated as ``$E1$-property''.

\subsection{The restriction method}
Working towards the proof of Theorem~\ref{RLTheorem}, we first establish some 
reductions. We will use the following setup. Let~$V$ be an 
$\mathbb{R}G$-module, and let $\rho$ denote the representation of~$G$ 
afforded by~$V$. Let $n \in \GL(V)$ be of finite order normalizing~$\rho(G)$.

\begin{lem}
\label{InvariantSubmodule}
Let~$V_1$ denote an $n$-invariant $\mathbb{R}G$-submodule of~$V$. If~$(G,V_1)$ 
has the $E1$-property, then $(G,V,n)$ has the $E1$-property.
\end{lem}
\begin{proof}
Let~$\rho_1$ denote the representation of~$G$ afforded by~$V_1$, and let~$n_1$
denote the restriction of~$n$ to an automorphism of~$V_1$. Then 
$n_1 \in \GL(V_1)$ has finite order and normalizes $\rho_1(G)$. By assumption, 
there exists $g \in G$ and a non-trivial vector $v \in V_1$ fixed 
by~$\rho_1(g)n_1$. Thus $\rho(g)n$ has eigenvalue~$1$.
\end{proof}

\begin{lem}
\label{DirectSums}
Let~$S$ be an irreducible $\mathbb{R}G$-module of odd dimension such that 
$(G,S)$ has the $E1$-property. If $V$ is the direct sum of an odd number of 
copies of~$S$, then $(G,V)$ has the $E1$-property.
\end{lem}
\begin{proof}
Put $A := \langle \rho(G), n \rangle$. This is a finite subgroup of $\GL(V)$ 
and $V$ is an $\mathbb{R}A$-module in the natural way. Let~$V_1$ be an 
irreducible $\mathbb{R}A$-submodule of~$V$ of odd dimension, and let 
$S_1 \leq V_1$ be an irreducible $\mathbb{R}\rho(G)$-submodule of~$V_1$. 
Then~$V_1$ and~$S_1$ are absolutely irreducible by 
Lemma~\ref{AbsolutelyIrreducible}. The character of $\mathbb{C} \otimes S_1$ is 
$A$-invariant, as $nS_1$ is an irreducible $\mathbb{R}\rho(G)$-submodule of~$V$, 
and thus  isomorphic to~$S_1$. 

Since $A/\rho(G)$ is cyclic, the character of $\mathbb{C} \otimes S_1$ extends 
to~$A$; see \cite[Corollary~$11.22$]{Isaacs}. Moreover, all absolutely 
irreducible $\mathbb{R}A$-submodules of $\Ind_{\rho(G)}^A( S_1 )$ have dimension 
$\dim ( S )$; see \cite[Corollary~$6.17$ (Gallagher's theorem)]{Isaacs}. 
As~$V_1$ is isomorphic to one of these, we have $V_1 = S_1$. The claim now 
follows from Lemma~\ref{InvariantSubmodule}. 
\end{proof}

\begin{lem}
\label{InvariantSubgroup}
Let $H \leq G$ be a subgroup of~$G$ such that~$n$ normalizes 
$\rho(H)$.  Let~$V_1$ denote an $n$-invariant $\mathbb{R}H$-submodule 
of~$\Res^G_H(V)$ such that $(H,V_1)$ has the $E1$-property.
Then $(G,V,n)$ has the $E1$-property.
\end{lem}
\begin{proof}
Apply Lemma~\ref{InvariantSubmodule} with $(G,V,n)$ replaced by 
$(H,\Res^G_H( V ),n)$. Write $\rho_H$ for the representation of~$H$ afforded by 
$\Res^G_H( V )$. Then~$n$ normalizes $\rho_H(H) = \rho(H)$ by assumption. 
Moreover,~$V_1$ is $n$-invariant and $(H,V_1)$ has the $E1$-pro\-per\-ty. By 
Lemma~\ref{InvariantSubmodule}, there is $g \in H$ such that 
$\rho_H(g) n = \rho(g) n$ has eigenvalue~$1$.
\end{proof}
 
\begin{cor}
\label{NormalSubgroups}
Suppose that~$V$ is irreducible and of odd dimension.

Let $H \unlhd G$ be a normal subgroup such that $\{ 1 \} \neq \rho(H)$ is 
characteristic in $\rho(G)$. Suppose that $(H,S)$ has the $E1$-property for
some irreducible submodule~$S$ of $\Res_H^G(V)$. Then $(G,V)$ has the 
$E1$-property.

If~$H$ has the $E1$-property, then $(G,V)$ has the $E1$-property.
\end{cor}
\begin{proof}
As $\rho(H) \neq \{ 1 \}$, the irreducible submodules of~$\Res_H^G(V)$ are 
non-trivial. As they are also odd-dimensional, the second statement follows 
from the first.

Write $\Res^G_H(V) = V_1 \oplus \cdots \oplus V_r$, where the $V_i$ are the
homogeneous components of $\Res^G_H(V)$. Chose the notation so that~$S$ is
a submodule of~$V_1$. As~$G$ permutes the~$V_i$ transitively, 
$\dim_{\mathbb{R}}( V_1 )$ is odd and there is $g \in G$ such that 
$\rho(g) n V_1 = V_1$. By Lemma~\ref{DirectSums} and our assumption, $(H,V_1)$ 
has the $E1$-property. Since $\rho(H)$ is characteristic in $\rho(G)$, the claim 
follows from Lemma~\ref{InvariantSubgroup} with $(G,V,n)$ replaced by 
$(G,V,\rho(g)n)$.
\end{proof}

\subsection{The minimal counterexamples}
Here we prove that a minimal counterexample to Theorem~\ref{RLTheorem}
is a non-abelian simple group.

\begin{prp}
\label{GroupDirectProduct}
Let~$H$ be a non-abelian finite simple group, and assume that $G = H \times 
\cdots \times H$ is a direct product of $r$ copies of~$H$. For each $1 \leq i 
\leq r$, let $V_i$ be an irreducible $\mathbb{R}H$-module of odd dimension.
Consider the $\mathbb{R}G$-module $V := V_1 \otimes \cdots \otimes V_r$ with
the $i$-th factor of~$G$ acting on~$V_i$.

Suppose that $V$ is not the trivial module and that for all $1 \leq i \leq r$
either $V_i$ is the trivial $\mathbb{R}H$-module or that $(H,V_i)$ has the
$E1$-property. Then $(G,V)$ has the $E1$-property.
\end{prp}
\begin{proof}
For $1 \leq i \leq r$, let $\rho_i \colon H \rightarrow \GL(V_i)$ denote the
representation of~$H$ afforded by~$V_i$. Then $\rho := \rho_1 \otimes \cdots 
\otimes \rho_r$ is the representation of $G$ afforded by~$V$. 

Let $n \in \GL(V)$ be of finite order normalizing~$\rho(G)$. For
$1 \leq i \leq r$, let $\nu_i \colon H \rightarrow G$ denote the embedding of~$H$
onto the $i$-th direct factor $H_i := \nu_i(H)$ of~$G$. As conjugation by~$n$
permutes the set $\{ \rho( H_i ) \mid 1 \leq i \leq r\}$ of normal subgroups
of $\rho(G)$, there is a permutation $\sigma$ of $\{ 1, \ldots , r \}$, and
$\alpha_i \in \Aut( H )$, $1 \leq i \leq r$, such that
\begin{multline}
\label{ConjugateTensors}
n \circ ( \rho_1( s_1 ) \otimes \cdots \otimes \rho_r( s_r ) ) \circ n^{-1}\\
= \rho_1( {{\alpha_1}( s_{\sigma^{-1}(1)} )} ) \otimes \cdots \otimes
\rho_r( {{\alpha_r}( s_{\sigma^{-1}(r)} )} )
\end{multline}
for all $s_1, \ldots , s_r \in H$.
Now consider the isomorphism of $\mathbb{R}$-vector spaces
$$f_\sigma \colon V_1 \otimes \cdots \otimes V_r \rightarrow V_{\sigma(1)} \otimes 
\cdots \otimes V_{\sigma(r)}, v_1 \otimes \cdots \otimes v_r \mapsto v_{\sigma(1)} 
\otimes \cdots \otimes v_{\sigma(r)}.$$
From Equation~(\ref{ConjugateTensors}) we obtain
\begin{multline}
\label{ConjugateTensors2}
f_\sigma \circ n \circ (\rho_1( s_1 ) \otimes \cdots \otimes \rho_r( s_r ) )
\circ n^{-1} \circ f_\sigma^{-1}\\
= (\rho_{\sigma(1)} \circ \alpha_{\sigma(1)})(s_1) \otimes \cdots \otimes
(\rho_{\sigma(r)} \circ \alpha_{\sigma(r)})(s_r)
\end{multline}
for all $s_1, \ldots , s_r \in H$.
Equation~(\ref{ConjugateTensors2}) shows that~$\rho$ is equivalent to the
representation
$$(\rho_{\sigma(1)} \circ \alpha_{\sigma(1)}) \otimes \cdots \otimes 
(\rho_{\sigma(r)} \circ \alpha_{\sigma(r)})\colon G \rightarrow \GL( V_{\sigma(1)} 
\otimes \cdots \otimes V_{\sigma(r)}).$$
As all the representations $\rho_i$ are absolutely irreducible by
Lemma~$2.5.1$, it follows that
$\rho_i$ is equivalent to $\rho_{\sigma(i)} \circ \alpha_{\sigma(i)}$ for all
$1 \leq i \leq r$. Thus there are $\mathbb{R}$-vector space isomorphisms
$a_i \colon V_i \rightarrow V_{\sigma(i)}$, such that
\begin{equation}
\label{ConjugateTensors3}
a_i \circ \rho_i( s ) \circ a_i^{-1} = (\rho_{\sigma(i)} \circ \alpha_{\sigma(i)})(s)
\end{equation}
for all $1 \leq i \leq r$ and all $s \in H$.
Equations~(\ref{ConjugateTensors2}) and~(\ref{ConjugateTensors3}) and the fact
that $\rho$ is absolutely irreducible now imply that $n = c f_\sigma^{-1} \circ 
(a_1 \otimes \cdots \otimes a_r)$ for some $0 \neq c \in \mathbb{R}$. Replacing
$a_1$ by $ca_1$ we may assume that $n = f_\sigma^{-1} \circ (a_1 \otimes \cdots 
\otimes a_r)$.

Suppose first that $\sigma$ is an $r$-cycle. Then all the $\rho(H_i)$ are 
isomorphic, so that, in particular,~$V_1$ is non-trivial. For $1 \leq i \leq r$
put $b_i := a_{\sigma^{i-1}(1)}$ and $\beta_i := \alpha_{\sigma^{i-1}(1)}$.
Using Equation~(\ref{ConjugateTensors3}), we find
$$
(b_r \circ \cdots \circ b_1) \circ \rho_1(s) \circ (b_r \circ \cdots \circ b_1)^{-1}
= (\rho_1 \circ \beta_1 \circ \beta_r \circ \cdots \circ \beta_2)(s)
$$
for all $s \in H$. Thus $\rho_1(H) \leq \GL(V_1)$ is invariant under
$b := b_r \circ \cdots \circ b_1 \in \GL(V_1)$. As $n = f_\sigma^{-1} \circ 
(a_1 \otimes \cdots \otimes a_r)$ has finite order, it follows that~$b$ has
finite order, since $(f_\sigma^{-1} \circ (a_1 \otimes \cdots \otimes a_r))^r = 
b \otimes c_2 \otimes \cdots \otimes c_r$ for suitable $c_i \in \Aut(V_i)$,
$2 \leq i \leq r$. As $(H,V_1)$ has the $E1$-property, there is
$0 \neq v_1 \in V_1$ and $s \in H$ such that $\rho_1(s)bv_1 = v_1$. For
$1 \leq i \leq r-1$, put $v_{\sigma^{i}(1)} := b_{i} v_{\sigma^{i-1}(1)}$. Then
\begin{eqnarray*}
\rho(s,1, \ldots , 1)n( v_1 \otimes \cdots \otimes v_r ) & = & 
\rho(s,1, \ldots , 1)(f_\sigma^{-1}(a_1 v_1 \otimes \cdots \otimes a_r v_r)) \\
& = & \rho_1(s)a_{{\sigma}^{-1}(1)}v_{{\sigma}^{-1}(1)} \otimes 
a_{{\sigma}^{-1}(2)}v_{{\sigma}^{-1}(2)} \otimes \\
& & \cdots \otimes a_{{\sigma}^{-1}(r)}v_{{\sigma}^{-1}(r)} \\
& = & v_1 \otimes \cdots \otimes v_r,
\end{eqnarray*} 
as 
$\rho_1(s)a_{{\sigma}^{-1}(1)}v_{{\sigma}^{-1}(1)} = 
\rho_1(s)b_rv_{{\sigma}^{r-1}(1)} = 
\rho_1(s)b_rb_{r-1}v_{{\sigma}^{r-2}(1)} = \cdots = \rho_1(s)bv_1 = v_1$, and
$a_{{\sigma}^{-1}(i)}v_{{\sigma}^{-1}(i)} = v_i$ for $2 \leq i \leq r$.

If $\sigma$ is not an $r$-cycle, we have tensor decompositions $V = W_1 \otimes 
W_2$ with $W_1 = V_1 \otimes \cdots \otimes V_{r'}$ and $W_2 = V_{r'+1} \otimes 
\cdots \otimes V_r$ for some $1 < r' < r$, and corresponding decompositions
$G = G_1 \times G_2$, $\rho = \mu_1 \otimes \mu_2$, $n = n_1 \otimes n_2$, where
each~$n_i$ has finite order and normalizes $\mu_i( G_i )$, for $i = 1, 2$.
Without loss of generality we can assume that~$W_1$ is non-trivial. If~$W_2$ is 
non-trivial or if $n_2 = \id_{W_2}$, arguing by induction on~$r$ we find 
elements $g_i \in G_i$ and non-zero vectors $w_i \in W_i$ such that 
$\mu_i( g_i )n_iw_i = w_i$ for $i = 1, 2$. If $W_2$ is the trivial module and
$n_2 = -\id_{W_2}$, then take $g_1 \in G_1$ and $0 \neq w_1 \in W_1$ such that
$\mu_1(g_1)(-n_1)w_1 = w_1$, and let $0 \neq w_2 \in W_2$. In both cases,
$0 \neq w_1 \otimes w_2$ is a fixed vector of $\rho( g_1, g_2 )n = 
\mu_1(g_1)n_1 \otimes \mu_2(g_2)n_2 = (-\mu_1(g_1)n_1) \otimes (-\mu_2(g_2)n_2)$.
\end{proof}

\begin{cor}
\label{E1ForDirectProducts}
Let~$H$ be a finite non-abelian simple group with the $E1$-property.
Then every finite direct product $G = H \times \cdots \times H$ has the 
$E1$-property.
\end{cor}
\begin{proof}
Suppose that $G$ is a direct product of~$r$ copies of~$H$.
Let~$V$ be a non-trivial irreducible $\mathbb{R}G$-module of odd dimension.
Then~$V$ is absolutely irreducible by Lemma~\ref{AbsolutelyIrreducible}.
Hence there are irreducible $\mathbb{C}H$-modules $V_i'$, $1 \leq i \leq r$
such that $\mathbb{C} \otimes V \cong V_1' \otimes \cdots \otimes V_r'$. For
$1 \leq i \leq r$, the isomorphism type of~$V_i'$ is uniquely determined by
the isomorphism type of~$V$, and thus the characters of the $V_i'$ are real 
valued. As $\dim( V_i' )$ is odd for $1 \leq i \leq r$, 
Lemma~\ref{AbsolutelyIrreducible} implies the existence of 
$\mathbb{R}H$-modules $V_i$ such that $\mathbb{C} \otimes V_i \cong V_i'$. 
Hence $V \cong V_1 \otimes \cdots \otimes V_r$ as $\mathbb{R}G$-modules. 
As~$V$ is non-trivial, at least one of the $V_i$ is non-trivial and thus 
$(H,V_i)$ has the $E1$-property by our assumption on~$H$. It follows from 
Proposition~\ref{GroupDirectProduct} that~$V$ has the $E1$-property.
\end{proof}

\begin{cor}
\label{ReductionCor}
A minimal counterexample to {\rm Theorem~\ref{RLTheorem}} is a 
non-abelian simple group.
\end{cor}
\begin{proof}
Let~$G$ be a group of minimal order without the $E1$-property.
Let~$V$ be a non-trivial irreducible $\mathbb{R}G$-module of odd dimension
such that $(G,V)$ does not have the $E1$-property, and let~$\rho$ denote the 
representation of~$G$ afforded by~$V$. Then~$\rho$ is faithful.

Let $H \unlhd G$ denote a non-trivial characteristic subgroup of~$G$. If 
$H \lneq G$, then~$H$ has the $E1$-property by assumption. 
Moreover, $\rho(H)$ is non-trivial and characteristic in $\rho(G)$,
as $\rho$ is faithful. But then $(G,V)$ has the $E1$-property 
by Corollary~\ref{NormalSubgroups}, contradicting our assumption.

Thus~$G$ is characteristically simple. In this case, 
Corollary~\ref{E1ForDirectProducts} implies that~$G$ is simple.
Lemma~\ref{OddDegreeClifford} and Example~\ref{TwoExamples}(b) imply 
that~$G$ is non-abelian.
\end{proof}

\begin{cor}
\label{SolvableGroups}
A solvable group has the $E1$-property.\hfill{\textsquare}
\end{cor}

\section{The $E1$-property for the simple sporadic and alternating groups}
\label{SporadicAndAlternating}

The aim of this section is to prove the $E1$-property for the simple sporadic
groups and the simple alternating groups. On the way to this, we establish 
further reductions.

\subsection{General notation}
We fix some pieces of notation that will be used throughout the remainder of
this article. 

\begin{notn}
\label{HypoSimple}
{\rm Let~$G$ be a non-abelian finite simple group. Let~$V$ denote a non-trivial 
irreducible $\mathbb{R}G$-module of odd dimension, and let~$\rho$ be the 
representation of~$G$ afforded by~$V$. Then~$\rho$ is faithful as~$G$ is simple. 
Moreover, we let $n \in \GL(V)$ denote an element of finite order 
normalizing~$\rho(G)$. Finally,~$\nu$ denotes the automorphism of~$G$ induced 
by~$n$, i.e.\ $\nu(g) = \rho^{-1}( {^n\!\rho(g)} )$ for $g \in G$.
}\hfill{\textsquare}
\end{notn}

Notice that if~$\chi$ denotes the character of~$G$ afforded by~$V$, then~$\chi$
is $\nu$-invariant.

\subsection{On the structure of the problem}
\label{PreliminaryConsiderations}

Assume Notation~\ref{HypoSimple}. In this subsection we will identify~$G$ with
its image $\rho(G) \leq \GL( V )$. Thus $G \leq \GL(V)$ is a non-abelian simple
group which acts absolutely irreducibly on~$V$. In particular, $C_{\GL(V)}( G ) 
= \{ x \cdot \id_V \mid 0 \neq x \in \mathbb{R} \}$. Also, $\nu = \ad_n$, and if
$n' \in \GL(V)$ is of finite order normalizing~$G$ and inducing the 
automorphism~$\nu$ of~$G$, then $n' = \pm n$, since 
$n^{-1}n' \in C_{\GL(V)}( G )$. As~$G$ is perfect, we have $G \leq \SL(V)$. 
Notice that $N_{\SL(V)}( G )$ embeds into~$\Aut(G)$, and thus every element of 
$N_{\SL(V)}( G )$ has finite order. As every element of finite order of~$\GL(V)$ 
has determinant~$1$ or~$-1$, it follows that the set of elements of finite order 
normalizing~$G$ is equal to the finite subgroup 
$N_{\SL(V)}(G) \times \langle -\id_V \rangle$ of $\GL(V)$.

\begin{defn}
\label{DefineA}
{\rm
Under the identification of~$G$ and~$\rho(G)$, set $A := \langle G, n \rangle$ 
and $A_1 := A \cap \SL(V)$. Thus $A_1 \leq A \leq \GL(V)$.
}\hfill{\textsquare}
\end{defn}

As~$A$ is a finite group, we have $C_A( G ) \leq \langle -\id_V \rangle$. 
We now distinguish two cases.

\medskip
\noindent
\textbf{Case 1.} Suppose that $- \id_V \not\in A$, so that, in particular,
$- \id_V \not\in \langle n \rangle$. Then $C_A( G ) = \{ \id_V \}$ and~$A$ 
embeds into~$\Aut(G)$. We get a chain of groups
$$G \leq A_1 \leq A \leq N_{\SL(V)}(G) \times \langle -\id_V \rangle.$$
This case occurs, e.g., for $G = A_5$ and $A \cong S_5$ when $\dim(V) = 5$. 

\medskip
\noindent
\textbf{Case 2.} Suppose that $- \id_V \in A$, so that 
$C_A( G ) = \langle -\id_V \rangle$. Since $-\id_V \not\in A_1$, we obtain
$$A = A_1 \times \langle -\id_V \rangle \leq 
N_{\SL(V)}(G) \times \langle -\id_V \rangle.$$
Thus $n = -n_1$ for some $n_1 \in A_1$. This case occurs, e.g., for 
$G = \SL_2( 8 )$ and $\dim(V) = 7$, where there exist an element 
$n \in \GL( V)$ of order~$6$ normalizing~$G$ such that 
$\langle G, n \rangle = N_{\SL(V)}( G ) \times \langle - \id_V \rangle$.

We record a simple consequence.

\begin{lem}
\label{CaseDistinction}
Assume Notation~{\rm \ref{HypoSimple}} and suppose we are in {\rm Case~$2$}. 
Then $A_1 = \langle G, n_1 \rangle$ and $|A_1/G|$ is odd. In particular, 
there is $g \in G$ such that $|gn_1|$ is odd.
\end{lem}
\begin{proof}
We have $A_1 = \langle G, n^2 \rangle = \langle G, n_1^2 \rangle 
\leq \langle G, n_1 \rangle \leq A_1$, and thus $\langle n_1G \rangle = A_1/G = 
\langle n_1^2G \rangle$, which implies that $|A_1/G|$ is odd. The last 
statement is clear, as the $2$-part of~$n_1$ lies in~$G$.
\end{proof}

It is also worthwhile to take a more abstract point of view.

\begin{defn}
\label{APrimeAndGPrime}
{\rm 
Set $A' := A$, respectively $A' := A_1$ if $(G,V,n)$ is as in Case~$1$, 
respectively Case~$2$.

Set $G' := \langle \Inn(G), \nu \rangle \leq \Aut(G)$.
}\hfill{\textsquare}
\end{defn}

\begin{lem}
\label{GPrimeAndA}
There is a surjective homomorphism 
$$\rho'\colon A \rightarrow G'$$
with
\begin{equation}
\label{GPrimeAndAFormula}
gn^i \mapsto \ad_g \circ \nu^i \text{\ for\ } g \in G \text{\ and\ } 
i \in \mathbb{Z}. 
\end{equation}
Moreover,~$\rho'$ restricts to an isomorphism $A' \rightarrow G'$.
\end{lem}
\begin{proof}
Let~$l$ denote the smallest positive integer such that $n^l \in G$. Then 
$\nu^l \in \Inn(G)$. Every element of~$A$ has a unique expression as 
$g n^i$ for some $g \in G$ and some integer~$i$ with $0 \leq i < l$. 
We can thus define a surjective map $\rho'\colon A \rightarrow G'$ 
by~(\ref{GPrimeAndAFormula}). Clearly,~$\rho'$ is a homomorphism with
kernel~$C_A( G )$. This proves our assertions.
\end{proof}

\begin{rem}
\label{VAsGPrimeModule}
{\rm
Let $\chi \in \Irr(G)$ and $\chi' \in \Irr(A)$ denote the irreducible characters
of~$G$, respectively~$A$, afforded by~$V$. We also write~$\chi'$ for the
restriction of~$\chi'$ to~$A'$. Thus $\chi' \in \Irr( A' )$ is an extension 
of~$\chi$. 

The isomorphism $(\rho'|_{A'})^{-1} \colon G' \rightarrow A'$ from 
Lemma~\ref{GPrimeAndA} makes~$V$ into an $\mathbb{R}G'$-module, and, by a slight 
abuse of notation, we also let~$\chi'$ denote the character of~$G'$ afforded 
by~$V$. Thus $\chi'( \rho'( a' ) ) = \chi'( a' )$ for all $a' \in A'$.
}\hfill{\textsquare}
\end{rem}

\subsection{The large degree method}
The following criterion is often helpful in small situations. Notice that
$G' = \langle \Inn(G), \ad_g \circ \nu \rangle$ for every $g \in G$.

\begin{lem}
\label{MainCriterion}
Suppose that there is $g \in G$ such that $\alpha := \ad_g \circ \nu$ has even 
order, and that, with the above notation, 
$$\Res^{G'}_{\langle \alpha \rangle}( \chi' )$$
contains each of the real, irreducible characters of $\langle \alpha \rangle$ 
with positive multiplicity. Then $(G,V,n)$ has the $E1$-property.
\end{lem}
\begin{proof}
Let $\rho'\colon A \rightarrow G'$ be the homomorphism from Lemma~\ref{GPrimeAndA}.
Then $\rho'( gn ) = \alpha$. In Case~$1$, our hypothesis shows that 
$\Res^{A'}_{\langle gn \rangle}( \chi' )$ contains the 
trivial character of~$\langle gn \rangle$, and thus~$gn$ has eigenvalue~$1$.

Suppose that we are in Case~$2$ and put $n_1 = -n \in A'$. Since 
$n^{-1}n_1 = -\id_V$, we have $\rho'( n_1 ) = \nu$ and thus 
$\rho'( gn_1 ) = \alpha$. By hypothesis, 
$\Res^{A'}_{\langle gn_1 \rangle}( \chi' )$ contains the non-trivial real
irreducible character of~$\langle gn_1 \rangle$, and thus~$gn_1$ has 
eigenvalue~$-1$. Hence $gn = -gn_1$ has eigenvalue~$1$.
\end{proof}

\begin{cor}
\label{MainCriterionCor}
Suppose that there is $g \in G$ such that $\alpha := \ad_g \circ \nu$ is an
involution. Then $(G,V,n)$ has the $E1$-property.
\end{cor}
\begin{proof}
As~$G'$ does not have a non-trivial abelian normal subgroup,~$\alpha$ neither is 
in the kernel nor in the center of~$\chi'$. Hence he hypothesis of 
Lemma~\ref{MainCriterion} is satisfied.
\end{proof}

We close by showing that if $\dim( V )$ is large relative to certain subgroups 
of~$G$, then $(G,V,n)$ has the $E1$-property. We will use the following 
notation. If $\alpha \in \Aut(G)$, and if~${p}$ is a prime dividing~$|\alpha|$, 
we write~$\alpha_{(p)}$ for an element of order~${p}$ in $\langle \alpha \rangle$.
\begin{lem}
\label{LargeDegrees}
Assume {\rm Notation~\ref{HypoSimple}}. Suppose that there is $g \in G$ such 
that $\alpha := \ad_g \circ \nu$ has even order and that
$$\dim(V) > (|\alpha| - 1)|C_{G}( \alpha_{(p)} )|^{1/2},$$
for all primes~${p}$ with ${p} \mid |\alpha|$.

Then $(G,V,n)$ has the $E1$-property.
\end{lem}
\begin{proof}
We will make use of Lemma~\ref{AllConstituents} for the inclusion
$\Inn(G) \unlhd G'$. Notice that $|C_{\text{\rm Inn}(G)}( \beta )| = |C_G( \beta )|$ for
every $\beta \in G'$. For every $1 \neq \beta \in \langle \alpha \rangle$, 
there is a prime~${p}$ dividing~$|\alpha|$ such that $C_G( \beta ) \leq 
C_G( \alpha_{(p)} )$. By hypothesis, 
$$\chi'(1) > (|\alpha|-1) |C_G(\beta)|^{1/2}$$
for all $1 \neq \beta \in \langle \alpha \rangle$.
Hence $\Res^{G'}_{\langle \alpha \rangle}( \chi' )$ contains all irreducible
characters of $\langle \alpha \rangle$ with positive multiplicity by
Lemma~\ref{AllConstituents}. The assertion follows from 
Lemma~\ref{MainCriterion}.
\end{proof}

\subsection{Some special cases}
We prove the $E1$-property for non-abelian simple groups with special
automorphism groups. 

\begin{lem}
\label{AutGEqualG}
If $\nu \in \Inn(G)$, then $(G,V,n)$ has the $E1$-property. In particular,~$G$ 
has the $E1$-property if $\Aut(G) = \Inn(G)$. 
\end{lem}
\begin{proof}
If $\nu \in \Inn(G)$, there is $g \in G$ such that $\alpha = \ad_g \circ \nu$ is 
an involution. The claim follows from Corollary~\ref{MainCriterionCor}.
\end{proof}

\begin{lem}
\label{SplitOfOrder2}
If $\Aut(G)$ is a split extension of~$\Inn(G)$ with
a group of order~$2$, then~$G$ has the $E1$-property. 
\end{lem}
\begin{proof}
By hypothesis, there is $g \in G$ such that 
$\alpha = \ad_g \circ \nu$ is an involution.
We are done by Corollary~\ref{MainCriterionCor}.
\end{proof}

\begin{cor}
\label{AlternatingAndSporadicGroups}
If $G = A_n$, the alternating group on $n$-letters with $n \neq 6$, or if~$G$ 
is a sporadic simple group, then~$G$ has the $E1$-property.
\end{cor}
\begin{proof}
It is well known that $\Aut(A_n) = S_n$, the symmetric group on~$n$ letters, 
unless $n = 6$, so the result follows from Lemma~\ref{SplitOfOrder2} in 
these cases.

If $G$ is a sporadic simple group, then either $\Aut(G) \cong G$ or $\Aut(G)$ is 
a split extension of~$\Inn(G)$ with a group of order~$2$; see \cite{Atlas}. We 
are done with Lemmas~\ref{AutGEqualG} and~\ref{SplitOfOrder2}
\end{proof}

\medskip
\noindent
The group $G = A_6$ excluded in Corollary~\ref{AlternatingAndSporadicGroups} 
will be treated as $G = \PSL_2(9)$.

\section{Simple groups of Lie type}
\label{SimpleGroupsOfLieType}

Here we introduce the simple groups of Lie type and some of their properties 
relevant to our investigations.

\subsection{The groups}
\label{TheGroups}
Let~$G$ be a finite simple group of Lie type defined over a field of 
characteristic~$r$. For a concise introduction to these groups see
\cite[Section~$2.2$]{GLS}, for a thorough treatment refer to~\cite{C1}. In the 
finitely many cases where~$r$ is not uniquely
determined by~$G$ (cf.\ \cite[Theorem~$2.2.10$]{GLS}), we choose~$r$ to be odd. 
Every finite simple group of Lie type is isomorphic to exactly one of the 
groups listed in Table~\ref{TableOfGroups}. In this table,~$q$ can be any power 
of~$r$, subject to the conditions given in the last column of the table. The 
groups in Lines~$1$--$6$ are the classical groups, and we give their classical
names as well as their ``Lie type'' names. For simplicity, we call the other 
groups exceptional groups of Lie type. The groups in Line~$14$ are the Suzuki 
groups, and the groups in Lines $15$,~$16$ the Ree groups. Finally the group in 
Line~$17$ is known as the Tits group. The groups in Lines~$2$,~$6$ and
$12$--$17$ are called twisted groups, the others are the untwisted groups.
Finally, in Lines~$1$--$13$, we let the positive integer~$f$ be such that 
$q = r^f$. Our notation for the Suzuki and Ree groups differs from the one used 
in~\cite{GLS}.
What we write as ${^2\!B}_2(q)$, ${^2G}_2(q)$ and ${^2\!F}_4(q)$, is written as
${^2\!B}_2(\sqrt{q})$, ${^2G}_2(\sqrt{q})$, respectively ${^2\!F}_4(\sqrt{q})$,
with $q = 2^{2m+1}$, $3^{2m+1}$, respectively $2^{2m+1}$; see 
\cite[Definition~$2.2.4$]{GLS}.

\setlength{\extrarowheight}{0.5ex}
\begin{table}
\caption{\label{TableOfGroups} The simple groups of Lie type}
$$
\begin{array}{rrrc}\hline\hline
\text{\rm Row} & \multicolumn{1}{c}{\text{\rm Names}} & \text{\rm Rank} & \text{\rm Conditions} \\ \hline\hline
1 & A_{d-1}(q), \PSL_{d}(q) & d \geq 2 & (d,q) \neq (2,2), (2,3), (2,4), (3,2) \\ 
2 & {^2\!A}_{d-1}(q), \PSU_{d}(q) & d \geq 3 & (d,q) \neq (3,2), (4,2) \\ 
3 & B_d(q), \P\Omega_{2d+1}(q) & d \geq 2 & (d,q) \neq (2,2) \\ 
4 & C_d(q), \PSp_{2d}(q) & d \geq 3 & q \text{\rm\ odd} \\ 
5 & D_d(q), \P\Omega^+_{2d}(q) & d \geq 4 &  \\ 
6 & {^2\!D}_d(q), \P\Omega^-_{2d}(q) & d \geq 4 &  \\ 
7 & G_2(q) & & q \geq 3 \\
8 & F_4(q) & & \\
9 & E_6(q) & & \\
10 & E_7(q) & & \\
11 & E_8(q) & & \\
12 & {^3\!D}_4(q) & & \\
13 & {^2\!E}_6(q) & & \\
14 & {^2\!B}_2(q) & & q = 2^{2m+1} > 2 \\
15 & {^2G}_2(q) & & q = 3^{2m+1} > 3 \\
16 & {^2\!F}_4(q) & & q = 2^{2m+1} > 2 \\
17 & {^2\!F}_4(q)' & & q = 2 \\ \hline
\end{array}
$$
\end{table}

\subsection{A $\sigma$-setup for~$G$} \label{ASigmaSetupForG}
It is convenient to introduce a $\sigma$-setup $(\overline{G}, \sigma)$ for~$G$; 
see \cite[Definition~$2.2.1$]{GLS}. We choose~$\overline{G}$ as a simple, adjoint 
algebraic group over the algebraic closure~$\mathbb{F}$ of the field with~$r$ 
elements, and~$\sigma$ as a suitable Steinberg morphism of~$\overline{G}$, 
such that $G = O^{r'}( \overline{G}^\sigma )$. If necessary, we will specify the
choice of~$\overline{G}$. Contrary to the usage in~\cite{GLS}, we will in general 
write $\overline{H}^\sigma$ rather than $C_{\overline{H}}( \sigma )$ for the set 
of $\sigma$-fixed points of a $\sigma$-stable subgroup~$\overline{H}$ 
of~$\overline{G}$. 

\subsection{The $BN$-pair and the Weyl group}
\label{BNPair}
Here, we mainly follow \cite[Chapters~$1$,~$2$]{GLS}. See also
\cite[Subsections~$8.5$ and~$13.5$]{C1} and 
\cite[Subsections~$2.5$ and~$2.6$]{C2}. 

The algebraic group~$\overline{G}$ has a $BN$-pair 
$(\overline{B}, \overline{N})$, where~$\overline{B}$ is a $\sigma$-stable
Borel subgroup of~$\overline{G}$ whose unipotent radical~$\overline{U}$ is also
$\sigma$-stable. Moreover,~$\overline{B}$ contains a $\sigma$-stable maximal
torus~$\overline{T}$ such that $\overline{B} = \overline{T}\,\overline{U}$.
Finally, $\overline{N} = N_{\overline{G}}( \overline{T} )$. We will fix such a
$BN$-pair and call $\overline{T}$ the standard (maximal) torus 
of~$\overline{G}$.
The pair $(\overline{T},\overline{B})$ gives rise to the root system 
$\Sigma = \Sigma( \overline{G} )$ of~$\overline{G}$, the set 
$\Pi \subseteq \Sigma$ of fundamental roots and the Dynkin diagram of~$\Sigma$.
The root system~$\Sigma$ is a subset of $X( \overline{T} ) = 
\Hom( \overline{T}, \mathbb{F}^* )$. For the possible Dynkin diagrams and their 
automorphisms see \cite[Subsection~$1.19$]{C2}.
The Weyl group of~$\overline{G}$ with respect to~$\overline{T}$ is denoted
by~$W( \overline{G} )$. By definition,
$W( \overline{G} ) = N_{\overline{G}}( \overline{T} )/\overline{T}$. Also,
$W( \overline{G} )$ can be identified with the Weyl group of~$\Sigma$ in a
natural way.

For every subset $I \leq \Pi$, there is a Levi subgroup $\overline{L}_I$
of~$\overline{G}$ contained in a parabolic subgroup $\overline{P}_I$, whose
unipotent radical is denoted by $\overline{U}_I$; see 
\cite[Theorem~$1.13.2$]{GLS}. In fact, 
$\overline{P}_I = \overline{U}_I\overline{L}_I$, a semidirect product. The 
Levi subgroups $\overline{L}_I$, where~$I$
runs through the subsets of~$\Pi$, are called the standard Levi subgroups
of~$\overline{G}$. A Levi subgroup of~$\overline{G}$ is a subgroup 
of~$\overline{G}$ which is conjugate in~$\overline{N}$ to a standard Levi 
subgroup. The Weyl group of a standard Levi subgroup~$\overline{L}_I$
is denoted by $W( \overline{L}_I )$. This is the parabolic subgroup
$W(\overline{G})_I$ of~$W(\overline{G})$ generated by the simple reflections
corresponding to the roots in~$I$.

We will assume that $\sigma$ is in standard form relative to~$\overline{B}$
and~$\overline{T}$, i.e.\ that~$\sigma$ satisfies the conditions 
of~\cite[Theorem~$2.2.3$]{GLS}.

The $BN$-pair of~$\overline{G}$ gives rise to a split $BN$-pair of~$G$ of 
characteristic~$r$; see \cite[Subsection~$2.5$]{C2} for the definition.
Put $B = \overline{B} \cap G$, $N = N_{\overline{G}}( \overline{T} ) \cap G$ and
$T = \overline{T} \cap G$; see \cite[Section~$2.3$]{GLS}. We call~$T$ the 
standard (maximal) torus of~$G$. We have $B = UT$ with 
$U := \overline{U}^{\sigma}$, the standard (maximal) unipotent subgroup of~$G$. 
We also have $\overline{G}^{\sigma} = G\overline{T}^{\sigma}$ (see
\cite[Theorem~$2.2.6$(g)]{GLS}), and $\overline{T}^{\sigma} \cap G = T$.
In particular, $\overline{G}^{\sigma}/G \cong \overline{T}^{\sigma}/T$.
Clearly,~$W(\overline{G})$ is $\sigma$-stable, and we 
write $W(\overline{G})^\sigma$ for its subgroup of $\sigma$-fixpoints. 
Then~$W(\overline{G})^\sigma$ is the Weyl group of the $BN$-pair of~$G$, i.e.\ 
$W(\overline{G})^\sigma = N/T$; see \cite[Theorem~$2.3.4$]{GLS}. 
Since~$\overline{T}$ is connected, we also have $N/T = W(\overline{G})^\sigma 
\cong N_{\overline{G}^{\sigma}}( \overline{T} )/\overline{T}^{\sigma}$. 
Since $\overline{T}$ and~$\overline{B}$ are $\sigma$-stable, the 
morphism~$\sigma$ acts on the root system~$\Sigma$ and fixes~$\Pi$. In 
particular,~$\sigma$ determines a symmetry~$\iota$ of the Dynkin diagram 
of~$\Pi$. 

A standard Levi subgroup $\overline{L}_I$ for $I \leq \Pi$ is~$\sigma$-stable,
if and only if~$I$ is fixed by~$\iota$. The standard Levi subgroups of~$G$ 
are the subgroups of the form $\overline{L}_I \cap G$, where~$I$ runs through 
the $\iota$-stable subsets of~$\Pi$. Let $I \subseteq \Pi$ be $\iota$-stable. 
We then put $L_I := \overline{L}_I \cap G$. The Weyl group of~$L_I$ equals 
$W( \overline{G} )_I^{\sigma}$; this is a parabolic subgroup of the Coxeter 
group~$W(\overline{G})^{\sigma}$. The standard parabolic 
subgroup~$\overline{P}_I$ of~$\overline{G}$ is also $\sigma$-stable, and we put 
$P_I := \overline{P}_I \cap G$. The unipotent radical of $P_I$ equals
$U_I := \overline{U}_I \cap G = \overline{U}^{\sigma}_I$. By definition, a Levi 
subgroup of~$G$ is a 
subgroup of~$G$ which is conjugate in~$N$ to a standard Levi subgroup. This is
sometimes also called a split Levi subgroup to distinguish it from the subgroups
of the form $\overline{L}^{\sigma}$, where~$\overline{L}$ is a $\sigma$-stable
Levi subgroup of~$\overline{G}$.

\subsection{Harish-Chandra theory}
\label{HarishChandraTheory}
We will apply Harish-Chandra theory to~$G$; see \cite[Chapters~$9$--$11$]{C2}
or~\cite[Chapter~$5$]{DiMi2}. 
If $P \leq G$ is a parabolic subgroup of~$G$ with Levi decomposition $P = UL$, 
where~$U$ denotes the unipotent radical of~$P$, we write $R^G_L$ and 
${^*\!R}^G_L$ for Harish-Chandra induction from~$L$, respectively Harish-Chandra 
restriction to~$L$. It is known that these operations on $\mathbb{C}G$-mod,
respectively $\mathbb{C}L$-mod, are independent of the chosen parabolic
subgroup of~$G$ containing~$L$ as a Levi complement. Moreover, these operations
are adjoint with respect to the standard scalar product on class function 
of~$G$, respectively~$L$. 

By definition, a principal series module of~$G$ is a composition factor of 
$R_T^G( \mathbb{C} )$, where~$\mathbb{C}$ denotes the trivial 
$\mathbb{C}T$-module. (We use the term principal series in the narrow sense to
mean the Harish-Chandra series corresponding to the cuspidal pair consisting of~$T$
and the trivial $\C T$-module.)
The principal series modules are labeled by the 
irreducible characters of $W(\overline{G})^{\sigma}$, and we usually denote 
a principal series module of~$G$ by the corresponding irreducible character 
of~$W(\overline{G})^{\sigma}$ or the label of this character. An irreducible 
$\mathbb{R}G$-module~$V'$ of odd dimension occurs in $R_T^G( \mathbb{R} )$ as a
composition factor, if and only if $\mathbb{C} \otimes V'$ is a principal series 
module; in this situation, we also call~$V'$ a principal series module. 

Principal series modules are preserved by Harish-Chandra induction and 
restriction, in the sense that all composition factors of a Harish-Chandra
induced principal series module are principal series modules; the 
analogous statement holds for Harish-Chandra restriction. For these statements
see \cite[Proposition~$5.3.9$]{DiMi2}.

A standard Levi subgroup~$L_I$ determines a parabolic 
subgroup~$W( \overline{G} )_I^{\sigma}$ of~$W( \overline{G} )^{\sigma}$. By the 
Howlett-Lehrer comparison theorem \cite[Theorem 5.9]{HowLeh2}, Harish-Chandra 
induction $R_{L_I}^G$ and restriction ${^*\!R}_{L_I}^G$ correspond to 
$\Ind_{W( \overline{G} )_I^{\sigma}}^{W( \overline{G} )^{\sigma}}$, respectively 
$\Res_{W( \overline{G} )_I^{\sigma}}^{W( \overline{G} )^{\sigma}}$. Taking 
$I = \emptyset$, we have $P_I = B$ and $L_I = T$. Thus, as already noted above, 
the composition factors of $R_T^G( \mathbb{C} )$ correspond to the 
irreducible constituents of $\Ind_{\{1\}}^{W( \overline{G} )^{\sigma}}( 1 )$, 
i.e.\ to the irreducible characters of~$W( \overline{G} )^{\sigma}$, and if the 
principal series module~$V'$ corresponds to 
$\vartheta \in \Irr( W( \overline{G} )^{\sigma} )$, then the multiplicity 
of~$V'$ as a composition factor of $R_T^G( \mathbb{C} )$ equals~$\vartheta(1)$.

The following lemma provides the base for the applications of Harish-Chandra
theory to our problem.
\begin{lem}
\label{MainHCLemma}
Let~$G$ be as in {\rm Subsection~\ref{TheGroups}}. Let $(V,n,\nu)$ be as in
{\rm Notation~\ref{HypoSimple}} and let~$\chi$ denote the character of~$V$. 

Suppose that there is a $\nu$-stable, proper parabolic subgroup~$P$ with a 
$\nu$-stable Levi complement~$L$ and there is $\psi \in \Irr(L)$ real, of odd 
degree, $\nu$-invariant and non-trivial, such that~$\chi$ occurs with odd 
multiplicity in $R_L^G( \psi )$. Then if~$L$ has the $E1$-property, so 
does~$(G,V,n)$. 
\end{lem}
\begin{proof}
As in Subsection~\ref{PreliminaryConsiderations}, we will identify~$G$ with
is image in~$\GL(V)$. Then~$P$ and~$L$ are $n$-invariant by assumption.

By hypothesis, $\langle R_L^G( \psi ), \chi \rangle = 
\langle \psi, {^*\!R}_L^G( \chi ) \rangle$ is odd. As~$\psi$ is real and of odd 
degree, Lemma~\ref{AbsolutelyIrreducible} implies the existence of an irreducible 
$\mathbb{R}P$-submodule~$S$ of~$\Res^G_P( V )$ such that the unipotent 
radical~$U$ of~$P$ acts trivially on~$S$, and the character of~$S$, viewed as 
an $\mathbb{R}L$-module, equals~$\psi$. Moreover, the $S$-homogeneous 
component~$V_1$ of~$\Res^G_P( V )$ has odd dimension.

Since~$\psi$ is $\nu$-invariant,~$nS$ is isomorphic to~$S$ as 
$\mathbb{R}L$-module. As~$U$ is $n$-invariant,~$U$ acts trivially on~$nS$. 
Thus $nS \cong S$ as $\mathbb{R}P$-modules and hence $nS \leq V_1$. It follows 
that~$V_1$ is $n$-invariant. As~$S$ is non-trivial and~$L$ has the 
$E1$-property, $(L, \Res^P_L( V_1 ) )$ has the $E1$-property by 
Lemma~\ref{DirectSums}. But then $(P,V_1)$ also has the $E1$-property.
Lemma~\ref{InvariantSubgroup} completes our proof.
\end{proof}

\subsection{Automorphisms}
\label{AutomorphismsI}
The automorphisms of~$G$ are described in \cite[Section~$2.5$]{GLS}. According 
to \cite[Theorem~$2.5.12$]{GLS}, we have
$\Aut(G) = \Inndiag(G) \rtimes (\Gamma_G\Phi_G)$, where~$\Inndiag(G)$ consists
of the automorphisms of~$G$ induced by conjugation with elements 
of~$\overline{G}^{\sigma}$, so that
$\Inndiag(G) \cong \Inn( \overline{G}^{\sigma} ) \cong \overline{G}^{\sigma}$.  

For some small cases when~$r$ is odd, but more notably when $r = 2$, we 
will need a more precise description of $\Aut(G)$. For simplicity, we only 
consider the groups of Lines~$1$--$12$ of Table~\ref{TheGroups}, and if~$G$
is one of $B_2(q)$, $G_2(q)$ or $F_4(q)$, we assume that $r \neq 2, 3, 2$ in the
respective cases. The cases not treated here will be discussed when they occur.
With these restrictions, we have $\Gamma_G\Phi_G = \Gamma_G \times \Phi_G$,
and~$\Gamma_G$ is isomorphic to the group of symmetries of the Dynkin diagram
of~$\overline{G}$, if~$G$ is as in one of the Lines~$1$,~$5$ or~$9$ of 
Table~\ref{TheGroups}, i.e.\ if~$G$ is one of the groups $\PSL_{d}(q)$,
$\P\Omega^+_{2d}(q)$ or~$E_6(q)$. Moreover, $|\Gamma_G| = 2$ in these cases, 
except for $G = \P\Omega^+_{4}(q)$, in which case~$\Gamma_G$ is isomorphic to
the symmetric group on three letters; see the corresponding Dynkin diagrams 
displayed in Figure~\ref{DynkinDiagramA}. In all other cases,~$\Gamma_G$ is 
trivial (under the restrictions on~$G$ imposed at the beginning of this 
paragraph).  Also,~$\Phi_G$ is cyclic, and we have $|\Phi_G| = f$, if~$G$ is 
untwisted, and $|\Phi_G| = 2f$ if $G$ is twisted, unless $G = {^3\!D}_4(q)$, 
in which case $|\Phi_G| = 3f$.

\begin{figure}[t] 
\caption{\label{DynkinDiagramA} The Dynkin diagrams of $A_{d-1}$,
$E_6$ and $D_4$}
\begin{center}
\begin{picture}(123.5, 40)(20,-10)
\put(  -60,   0){$A_{d-1}$:}
\put(    0,   0){\circle{7}} 
\put(   40,   0){\circle{7}} 
\put(   80,   0){\circle{7}} 
\put(  160,   0){\circle{7}} 
\put(  200,   0){\circle{7}}
\put(  3.5,   0){\line(1,0){ 33}}
\put( 43.5,   0){\line(1,0){ 33}}
\multiput( 88.5,   0)(6,0){11}{\line(1,0){ 3}}
\put(163.5,   0){\line(1,0){ 33}}
\put(   -3,+8){$1$}
\put(   37,+8){$2$}
\put(   77,+8){$3$}
\put(  146.5,+8){$d-2$}
\put(  187.5,+8){$d-1$}
\end{picture}
\end{center}

\begin{center}
\begin{picture}(123.5, 90)(0,-60)
\put(  -80,   0){$E_{6}$:}
\put(    0,   0){\circle{7}}
\put(   40,   0){\circle{7}}
\put(   80,   0){\circle{7}}
\put(  120,   0){\circle{7}}
\put(  160,   0){\circle{7}}
\put(   80, -40){\circle{7}}
\put(  3.5,   0){\line(1,0){ 33}}
\put( 43.5,   0){\line(1,0){ 33}}
\put( 83.5,   0){\line(1,0){ 33}}
\put(123.5,   0){\line(1,0){ 33}}
\put(80,   -3.5){\line(0,-1){ 33}}
\put(   -3,+8){$1$}
\put(   37,+8){$3$}
\put(   77,+8){$4$}
\put(  117,+8){$5$}
\put(  157,+8){$6$}
\put(   77,-56){$2$}
\end{picture}
\end{center}

\begin{center}
\begin{picture}(123.5, 90)(-80,-40)
\put( -160,   0){$D_{4}$:}
\put(  -18, 32,64){\line(1, -2){15}}
\put(  -18,-32,64){\line(1,  2){15}}
\put(  -20, 34,64){\circle{7}}
\put(  -20,-34,64){\circle{7}}
\put(    0,   0){\circle{7}}
\put(   40,   0){\circle{7}}
\put(  3.5,   0){\line(1,0){ 33}}
\put(  -33, 32,64){$1$}
\put(  -33,-39,64){$2$}
\put(   -3,+8){$3$}
\put(   37,+8){$4$}
\end{picture}
\end{center}
\end{figure}
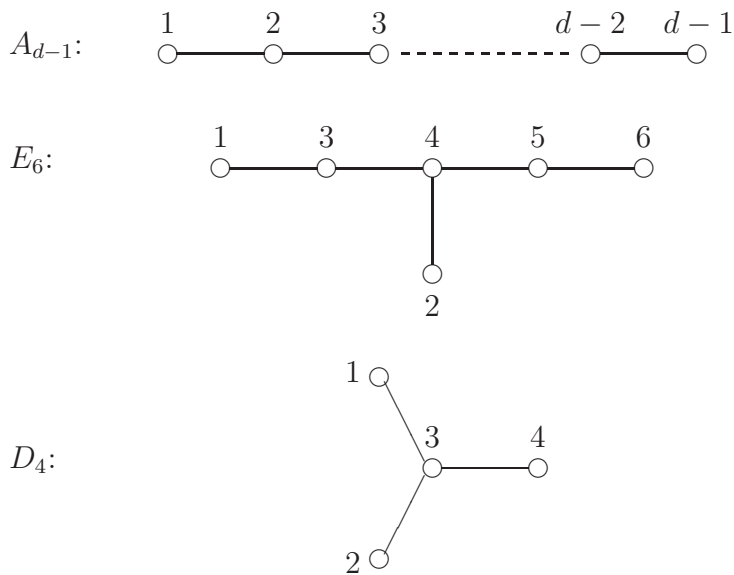

Let us recall further terminology concerning automorphisms of~$G$, following 
\cite[Definition~$2.5.13$]{GLS}. First, we choose a particular generator of 
$\Phi_G$ and a particular element of~$\Gamma_G$. Let $\varphi := \varphi_r$ 
denote the standard Frobenius endomorphism of~$\overline{G}$ introduced in 
\cite[Theorem~$1.15.4$(a)]{GLS}, and put $\Phi_{\overline{G}} := 
\langle \varphi \rangle$ as in \cite[Definition~$1.15.5$(a)]{GLS}. The 
restriction of~$\varphi$ to~$G$, also 
denoted by~$\varphi$, is our preferred generator of~$\Phi_G$. Suppose 
that~$\iota$ is a non-trivial symmetry of the Dynkin diagram of~$\overline{G}$. 
We then denote, by the same letter, the element of~$\Gamma_{\overline{G}}$ 
introduced in \cite[Theorem~$1.15.2$(a), Definition~$1.15.5$(e)]{GLS}, as well 
as the restriction of the latter to~$G$. Such a~$\iota$ will be called a
standard graph automorphism of~$\overline{G}$, respectively~$G$. Notice that
the groups 
$\overline{B}^{\sigma}$, $\overline{T}^{\sigma}$ and $\overline{N}^{\sigma}$, 
and thus also,~$B$,~$T$ and~$N$ are fixed, up to inner automorphisms 
of~$\overline{G}^{\sigma}$, by $\Aut(G)$. Finally, we may choose notation so
that $\sigma = \varphi^f$, if~$G$ is untwisted, and $\sigma = \iota \circ 
\varphi^{f}$ for some $1 \neq \iota \in \Gamma_{\overline{G}}$, otherwise.

Let $\alpha \in \Aut(G)$. Then~$\alpha$ is of the form
$$
\alpha = \ad_g \circ \iota \circ \varphi^b
$$
with $g \in \overline{G}^{\sigma}$, $\iota \in \Gamma_G$ and
$0 \leq b < |\Phi_G|$. 
Then~$\alpha$ is an \textit{inner-diagonal
automorphism} of~$G$ if and only if $\iota = 1$ and $b = 0$. Suppose now 
that~$\Gamma_G$ is non-trivial. If $g = 1$ and $\iota = 1$, any 
$\Aut(G)$-conjugate of~$\alpha$ is a \textit{field automorphism} of~$G$. 
If $\iota \neq 1$ and $b = 0$, any 
$\Aut(G)$-conjugate of~$\alpha$ is a \textit{graph automorphism} of~$G$. Suppose 
now that~$\Gamma_G$ is trivial. If $g = 1$, any $\Aut(G)$-conjugate of~$\alpha$ 
is a \textit{field automorphism} of~$G$, provided $|\alpha|$ is odd if~$G$ is 
twisted and $G \neq {^3\!D}_4(q)$, respectively $3 \nmid |\alpha|$ if 
$G = {^3\!D}_4(q)$. Finally, if~$G$ is twisted,~$\alpha$ is a 
\textit{graph automorphism} of~$G$, if $|\varphi^b|$ is even, respectively
divisible by~$3$ in case $G = {^3\!D}_4(q)$.

\section{Simple groups of Lie type of odd characteristic}
\label{OddCharacteristic}

In this section we let~$G$ be one of the simple groups of Lie type listed in 
Table~\ref{TableOfGroups}, where we assume that $q = r^f$ is odd. We use the 
terminology and the concepts introduced in Section~\ref{SimpleGroupsOfLieType}. 
We also assume Notation~\ref{HypoSimple}. In particular,~$V$ is a non-trivial 
irreducible $\mathbb{R}G$-module of odd dimension,~$\rho$ is the corresponding 
homomorphism $G \rightarrow \GL(V)$, and~$n$ is an element of $\GL(V)$ of finite 
order normalizing~$\rho(G)$. Moreover,~$\nu$ denotes the automorphism of~$G$
induced by~$n$. We say that~$V$ is the Steinberg module of~$G$, if 
$\mathbb{C} \otimes V$ affords the Steinberg representation of~$G$. The main 
goal in this section is to show that the groups considered here are not minimal 
counterexamples to Theorem~\ref{RLTheorem}.

\subsection{Some special cases}
We first deal with the Steinberg modules in some groups of small rank.

\begin{lem}
\label{SmallRankOddCharacteristic}
Let 
$$G \in \{ \PSL_2( q ), \PSL_3( q ), \PSU_3( q ), G_2( 3^{f} ),
{^2G}_2( 3^{2m+1} ) \}$$ 
and suppose that~$V$ is the Steinberg module of~$G$. Then $(G,V)$ has the 
$E1$-property.
\end{lem}
\begin{proof}
For some of the proofs below we rely on Lemma~\ref{LargeDegrees}. For this, we 
need to estimate the order of subgroups $C_G( \beta )$ for certain automorphisms 
$\beta \in \Aut(G)$, for which we cite \cite[Propositions~$4.9.1$,~$4.9.2$]{GLS}. 
However, \cite[Propositions~$4.9.1$]{GLS} only gives $O^{r'}( C_G( \beta ) )$.
Nevertheless, in combination with the tables of maximal subgroups determined 
in~\cite{BHRD}, we obtain the desired bounds.

Let us begin with $G = \PSL_2(q)$, with $q = r^f > 3$ odd. The elements of~$G$
are written as
$$\left[ \begin{array}{cc} a & b \\ c & d \end{array} \right],$$
where the square brackets indicate the image in $\PSL_2(q)$ of the
matrix 
$$\left( \begin{array}{cc} a & b \\ c & d \end{array}\right) \in \SL_2(q).$$
Let us define two automorphisms of~$G$. First, choose $\zeta \in \mathbb{F}_q$
of order $(q-1)_2$. Let $\delta$ denote conjugation by the diagonal
matrix $$\left[ \begin{array}{cc} \zeta & 0 \\ 0 & 1 \end{array}\right],$$
i.e.\ 
$$\delta\left( \left[ \begin{array}{cc} a & b \\ c & d \end{array} \right] \right) = 
\left[ \begin{array}{cc} a & \zeta b \\ \zeta^{-1} c & d \end{array} \right].$$
Next, $\varphi := \varphi_r$ denotes the standard Frobenius morphism of~$G$ of 
order~$f$, i.e.\ 
$$\varphi\left( \left[ \begin{array}{cc} a & b \\ c & d \end{array} \right] \right) =
\left[ \begin{array}{cc} a^r & b^r \\ c^r & d^r \end{array} \right].$$
Then $\varphi \circ \delta = \delta^r \circ \varphi$.
By the results reported in Subsection~\ref{AutomorphismsI}, every automorphism 
of~$G$ is of the form $\ad_g \circ \delta^k \circ \varphi^l$, for some $g \in G$ 
and some integers~$k$ and~$l$. 

Let~$B$ and~$T$ denote the images of the subgroups of upper triangular matrices,
respectively diagonal matrices, of $\SL_2(q)$ in $\PSL_2(q) = G$. We take~$B$ 
and~$T$ as our standard Borel subgroup, respectively standard maximal torus.
Then~$B$ and~$T$ are invariant under~$\delta$ and~$\varphi$. A set of 
representatives for the left cosets of~$B$ in~$G$ is given by
$\mathcal{R} = \mathcal{R}_q \cup \mathcal{R}_{\infty}$ with
$$\mathcal{R}_q := \left\{ \left[ \begin{array}{cc} 1 & 0 \\ y & 1 \end{array}\right] 
\mid y \in \mathbb{F}_q \right\}$$
and
$$\mathcal{R}_{\infty} := 
\left\{ \left[ \begin{array}{cc} 0 & -1 \\ 1 & 0 \end{array} \right] \right\}.$$
(These elements may be identified with the projective line over $\mathbb{F}_q$.)

By replacing~$n$ with a suitable element of its coset $Gn$, we may assume that
$\nu = \delta^k \circ \varphi^l$ for some integers~$k$ and~$l$. Then~$\nu$ 
stabilizes~$B$ and permutes the elements of~$\mathcal{R}_q$.
Now identify~$G$ with the image of~$\rho$ in~$\GL(V)$ and adopt the notation 
introduced in Subsection~\ref{PreliminaryConsiderations}. In 
particular, $A = \langle G, n \rangle$; see Definition~\ref{DefineA}.

Suppose now that we are in Case~$1$ of 
Subsection~\ref{PreliminaryConsiderations}. By Lemma~\ref{GPrimeAndA}, there is
an isomorphism $A \rightarrow G' = \langle \Inn(G), \nu \rangle$, sending~$n$
to~$\nu$. In particular, $\langle n \rangle \cong \langle \nu \rangle$.
As~$B$ is $\nu$-invariant, $\langle B, n \rangle$ is an $\ad_n$-invariant 
subgroup of~$A$. Since~$B$ 
is a maximal subgroup of~$G$, we have $\langle B, n \rangle \cap G = B$. 
Thus~$\mathcal{R}$ is a set 
of representatives for the left cosets of~$\langle B, n \rangle$ in~$A$. 
Notice that
$$nx\langle B, n \rangle = nx\langle B, n \rangle n^{-1} = 
nxn^{-1}\langle B, n \rangle$$
for $x \in \mathcal{R}$. As $nxn^{-1} = \nu(x)$ for $x \in G$, left 
multiplication by~$n$ fixes the coset with representative in 
$\mathcal{R}_{\infty}$, and the action of $\langle n \rangle$ by left 
multiplication on the set of $\langle B, n \rangle$-cosets corresponding 
to~$\mathcal{R}_q$ is equivalent to the action of~$\langle \nu \rangle$ 
on~$\mathcal{R}_q$. Thus, left multiplication by~$n$ also fixes the coset 
corresponding to the trivial element in~$\mathcal{R}_q$.

Consider the character 
$\psi := \Ind_{\langle B, n \rangle}^{A}( 1_{\langle B, n \rangle} )$. 
Thus~$\psi$ is the permutation character of~$A$ acting by left multiplication on 
the cosets of $\langle B, n \rangle$. By Mackey's theorem, 
$\Res_G^{A}( \psi ) = \Ind_B^G( 1_B ) = 1_G + \St_G$, where $\St_G$ denotes
the character of the Steinberg module of~$G$. Hence $\psi = 1_{A} +
\St_{A}$, where $\St_{A}$ is a rational valued extension of $\St_G$ to~$A$. 
As left multiplication by~$n$ fixes two $\langle B, n \rangle$-cosets, 
$\psi(n) \geq 2$ and so $\St_{A}(n)$ is a 
positive integer. The other extensions of~$\St_G$ to~$A$ are of the form
$\tilde{\lambda} \cdot \St_{A}$, where $\tilde{\lambda}$ is the inflation of 
some $\lambda \in \Irr( A/G )$ to~$A$. As~$A/G$ is cyclic, the only
real extensions of~$\St_G$ to~$A$ are of the form 
$\tilde{\lambda} \cdot \St_{A}$, where $\lambda \in \Irr(A/G)$ with 
$\lambda^2 = 1$. Viewing $\tilde{\lambda}$ as a character 
of~$\langle B, n \rangle$ by 
restriction, we get $\Ind_{\langle B, n \rangle}^{A}( \tilde{\lambda} ) =
\tilde{\lambda} + \tilde{\lambda} \cdot \St_{A}$.

Let $\lambda \in \Irr(A/G)$ with $\lambda^2 = 1$ such that 
$\chi' := \tilde{\lambda} \cdot \St_{A}$ is the character of~$A$ afforded 
by~$V$, where, again,~$\tilde{\lambda}$ denotes the inflation of~$\lambda$ to 
a character of~$A$. To show that~$n$ is represented by a matrix with 
eigenvalue~$1$, we have 
to show that $\Res^{A}_{\langle n \rangle}( \chi' )$ contains a trivial 
constituent.  Now, once more by Mackey's theorem, we have
\begin{equation}
\label{Mackey}
\Res_{\langle n \rangle}^{A}( \chi' ) = 
- \Res^{A}_{\langle n \rangle}( \tilde{\lambda} ) 
+ \sum_{z \in \mathcal{R}'} 
\Ind_{\langle n \rangle_z}^{\langle n \rangle}( \tilde{\lambda}_z ),
\end{equation}
where $\mathcal{R'} \subseteq \mathcal{R}$ is a set of representatives for the 
$\langle n \rangle$-$\langle B, n \rangle$-double cosets of~$A$, 
$\langle n \rangle_z := {^z\!\langle B, n \rangle} \cap \langle n \rangle$, and 
$\tilde{\lambda}_z := 
\Res_{\langle n \rangle_z}^{{^z\!\langle B, n \rangle}}( {^z\!\tilde{\lambda}} )$. 
If~$\tilde{\lambda}$ is the trivial character, every summand 
$\Ind_{\langle n \rangle_z}^{\langle n \rangle}( \tilde{\lambda}_z )$ 
of~(\ref{Mackey}) contains a trivial constituent. In this case 
$\Res_{\langle n \rangle}^{A}( \chi' )$ contains the trivial character, as 
$\langle n \rangle \langle B, n \rangle = \langle B, n \rangle \lneq A$, and thus
$|\mathcal{R}'| \geq 2$. Suppose that $\tilde{\lambda} \neq 1_{A}$. Then
$\Res^{A}_{\langle n \rangle}( \tilde{\lambda} ) \neq 1_{\langle n \rangle}$, as 
$\tilde{\lambda}(n) = -1$. Thus, in order to show that 
$\Res_{\langle n \rangle}^{A}( \chi' )$ contains a trivial character,
it suffices to show that there is $z \in \mathcal{R}'$ such that 
$\langle n \rangle_z$ is the trivial group. 

Observe that~$\langle n \rangle_z$ is the stabilizer in~$\langle n \rangle$ of 
the coset $z\langle B, n \rangle$. Thus $\langle n \rangle_z$ is trivial, if and 
only if~$z\langle B, n \rangle$ lies in a regular $\langle n \rangle$-orbit. 
This can only be the case if $z \in \mathcal{R}_q$, and we have to show that
$\mathcal{R}_q$ contains a regular $\langle \nu \rangle$-orbit. 
The non-trivial elements of~$\langle \nu \rangle$ are of the form 
$\delta^s \circ \varphi^t$ for integers~$s$ and~$t$ such that $\delta^s \neq 1$ 
or $\varphi^t \neq 1$. For $y \in \mathbb{F}_q$, let
$$z(y) := \left[ \begin{array}{cc} 1 & 0 \\ y & 1 \end{array}\right].$$ 
Then $\delta^s \circ \varphi^t(z(y)) = z(y')$ with
$y' = \zeta^{-s} y^{r^t}$. Thus $\delta^s \circ \varphi^t$ fixes $z(y)$
if and only if $y^{r^t-1} = \zeta^s$. If~$r$ is not a Mersenne prime or if 
$f > 2$, choose a primitive prime divisor $\ell$ of $q - 1 = r^f - 1$, i.e.\ 
$\ell \mid r^f - 1$ and $\ell \nmid r^e - 1$ for all $1 \leq e < f$; see
\cite[Theorem~IX.$8.3$]{HuBII}. Let $y \in \mathbb{F}_q^*$ denote an element of 
order~$\ell$. Then $y^{r^t - 1} \neq 1$, unless $t = f$. As $\zeta$ has even 
order, it follows that~$z(y)$ is only fixed by the trivial
element of~$\langle \nu \rangle$.

It remains to consider the case when~$r$ is a Mersenne prime and $f \leq 2$. 
Then $\Out(G)$ is an elementary abelian $2$-group of order~$4$, 
respectively~$2$. In particular, $\nu^2 \in \Inn(G)$. As we are in Case~$1$, 
this implies that $n^2 \in G$. As~$G$ is in the kernel of $\tilde{\lambda}$ by 
definition,~$n^2$ is in the kernel of ${^z\!\tilde{\lambda}}$ for all 
$z \in \mathcal{R'}$. Since~$\langle n \rangle \cong \langle \nu \rangle$ is a 
cyclic $2$-group, every 
proper subgroup of~$\langle n \rangle$ is contained in $\langle n^2 \rangle$. 
Thus if $\langle n \rangle_z \lneq \langle n \rangle$ for some 
$z \in \mathcal{R}'$, the corresponding summand of~(\ref{Mackey}) contains a 
trivial constituent. Since~$\langle n \rangle$ is not a normal subgroup of~$A$, 
it does not act trivially on the set of left cosets of~$\langle B, n \rangle$ 
in~$A$. Thus there is $z \in \mathcal{R}$ such that $z\langle B, n \rangle$ is 
not fixed by~$\langle n \rangle$. We choose~$\mathcal{R}'$ such that
$z \in \mathcal{R}'$. Then $\langle n \rangle_z \lneq \langle n \rangle$ and we 
are done.

Now assume that we are in Case~$2$ of 
Subsection~\ref{PreliminaryConsiderations}. With the notation introduced there,
we have $n = - n_1$ for some $n_1 \in A_1$. By Lemma~\ref{CaseDistinction} we may 
assume that~$|n_1|$ is odd. Lemma~\ref{GPrimeAndA} implies 
that~$|\nu|$ is odd, and so~$\nu$ is a field automorphism of~$G$.
More precisely, consider $K := \langle \delta, \varphi \rangle \leq \Aut(G)$.
This is a semidirect product 
$\langle \delta \rangle \rtimes \Phi_G$ with $\Phi_G = \langle \varphi \rangle$, 
and~$\Phi_G$ contains a $2'$-Hall subgroup of~$K$. Hence every element of odd 
order of~$K$ is conjugate in~$K$ to an element 
of~$\Phi_G$. In particular, $C_G( \nu' )$ for $\nu' \in \langle \nu \rangle$ is 
isomorphic to $C_G( \varphi' )$ for some $\varphi' \in \Phi_G$. Thus
$C_G( \nu )$ contains a subgroup isomorphic to $\PSL_2(p)$.

Let~$g$ be an involution in $C_G( \nu )$,
and put $\alpha := \ad_g \circ \nu$. Now make use of Lemma~\ref{LargeDegrees}.
If~$p$ is an odd prime diving $|\alpha|$, then $C_G( \alpha_{(p)} )$ is
isomorphic to $\PSL_2( q_0 )$, where $q_0^{p} = q$; see 
\cite[Proposition~$4.9.1$]{GLS} and \cite[Table~$8.1$]{BHRD}. In particular, 
$|C_G( \alpha_{(p)} )| \leq q_0(q_0^2 - 1) \leq q$, as~$p$ is odd. If $p = 2$, 
then $|C_G( \alpha_{(p)} )| = |C_G( g )| \leq q + 1$. If $f = 1$, then~$\nu$ is 
trivial and we are done with Corollary~\ref{MainCriterionCor}.
We will thus assume that $f > 1$, in which case $f \geq 3$, as
$|\nu| = |n_1|$ is odd. This implies that $\dim(V) = q > (2f - 1)(q + 1)^{1/2}$,
hence $(G,V,n)$ has the $E1$-property by Lemma~\ref{LargeDegrees}, as 
$|\alpha| \leq 2f$.

Let us now consider the case of $G = \PSL_3(q)$. The two standard Levi subgroups 
of~$G$ of type~$A_1$ are conjugate in~$G$. (This is not true for the
corresponding standard parabolic subgroups.) Let~$L$ be one of these. As~$\nu(L)$
is $G$-conjugate to a standard Levi subgroup, there is $g \in G$ such that
$\ad_g \circ \nu$ fixes~$L$. Again, identify~$G$ with its image $\rho(G)$ in
$\GL(V)$. Replacing~$n$ by~$gn$, we may assume that~$n$ normalizes~$L$.
Let~$\St_G$ and~$\St_L$ denote the Steinberg characters of~$G$ and~$L$, 
respectively. Since~$\St_L$ is invariant under every automorphism of~$L$ by 
\cite[Theorem~$2.5$]{MalleExt}, the homogeneous component~$V_1$ of 
$\Res^G_L( V )$ corresponding to~$\St_L$ is $n$-invariant. To show that $(G,V)$ 
has the $E1$-property, it suffices to show that $(L,V_1)$ has this property; see 
Lemma~\ref{InvariantSubgroup}. Using~\cite{SiFra} and~\cite{SteinGL} or 
\cite[Theorem~$6.5.9$]{C2}, one checks that
$$
\langle \Res^G_L(\St_G), \St_L \rangle = 
\begin{cases} 3, & \text{if\ } 3 \nmid q - 1 \\
              5, & \text{if\ } 3 \mid q - 1
\end{cases}
$$
By Lemma~\ref{DirectSums}, it suffices to show that a module affording~$\St_L$ 
has the $E1$-property. Now~$\St_L$ restricts to the Steinberg 
character~$\St_{L'}$ of $L' := [L,L] \cong \SL_2(q)$. By what we have already 
shown, a module affording~$\St_{L'}$ has the $E1$-property. Thus a module
affording~$\St_L$ has the $E1$-property by Corollary~\ref{NormalSubgroups}. This 
completes the proof in case $G = \PSL_3(q)$.

We now consider the case $G = \PSU_3( q )$, where, once more, we are going to 
apply Lemma~\ref{LargeDegrees}. Here, $\Aut(G) = \Inndiag(G) \rtimes \Phi_G$
with $\Inndiag(G) \cong \PGU_3(q)$ and~$\Phi_G$ cyclic of order~$2f$; see
Subsection~\ref{AutomorphismsI}. We assume that a generator~$\varphi$ 
of~$\Phi_G$ is the image of the standard Frobenius morphism of~$\SU_3(q)$ which
raises every entry of an element of~$\SU_3(q)$ to its $r$th power.
By pre-multiplying~$n$ with a suitable element 
of~$G$, we may assume that $\nu = \ad_t \circ \mu$ with $\mu \in \Phi_G$, 
and where $t \in \PGU_3(q)$ is represented by the matrix 
$\hat{t} = \diag(1,\zeta,1) \in \GU_3(q)$ with $\zeta \in \mathbb{F}_{q^2}^*$ of 
order dividing $(q+1)_3$. Notice that~$t$ is inverted by~$\varphi^f$.
If~$\nu$ has even order, put $g := 1$. If~$\nu$ has odd order, let $g \in G$
denote the image of $\diag(-1,1,-1) \in \SU_3(q)$ in~$G$. 
Let $\alpha := \ad_g \circ \nu$. Then~$\alpha$ has even order. 

Suppose first that $f \geq 2$. It follows from 
\cite[Propositions $4.9.1$, $4.9.2$]{GLS} in conjunction with
\cite[Tables~$8.5$,~$8.6$]{BHRD} that
$|C_G( \alpha_p )| \leq |\GU_2(q)| \leq q^4 + q^3$ for all primes~$p$ 
dividing~$|\alpha|$.  Also, $|\alpha| \leq 2f(q+1)_3$, as $\nu$ centralizes~$g$ 
and~$\langle \mu \rangle$ normalizes~$\langle t \rangle$. As~$r$ is odd and 
$f \geq 2$, there exists a primitive prime divisor~$\ell$ of 
$r^{2f} - 1$; see \cite[Theorem IX.$8.3$]{HuBII}. That is,~$\ell$ is a prime
with $\ell \mid r^{2f} - 1$ but $\ell \nmid r^j - 1$ for all $1 \leq j < 2f$. 
In particular, $\ell \mid q + 1$, and $2f \mid \ell - 1$. The latter implies 
that $\ell > 3$. As~$q$ is odd, $(q+1)_3 \mid (q+1)/(2\ell)$, and so 
$|\alpha| \leq 2f(q+1)_3 < (q+1)/2$. We conclude that
$$(|\alpha| - 1)^2(q^4 + q^3) < \frac{(q-1)^2}{4}(q^4 + q^3) =
\frac{1}{4}(q^6-q^5-q^4+q^3) < q^6.$$
Taking square roots we obtain
$$\dim(V) = q^3 > (|\alpha| - 1)|C_G( \alpha_p )^{1/2}$$
for all primes~$p$ dividing~$|\alpha|$. Lemma~\ref{LargeDegrees} implies
that~$(G,V,n)$ has the $E1$-property. Suppose now that $f = 1$.
Then $\Phi_G = \langle \varphi \rangle$ with $|\varphi| = 2$. If 
$t = 1$ or if $t \neq 1$ and $\nu = \ad_t \circ \varphi$, then 
$|\alpha| = 2$, hence $(G,V,n)$ has the $E1$-property by 
Corollary~\ref{MainCriterionCor}. It remains to consider the case
$t \neq 1$ and $\mu = \id_G$. Then $G' = \langle \Inn(G), \nu \rangle 
\cong \PGU_3( q )$, and $\alpha = \ad_{gt}$.
Let~$\chi'$ denote the character
of~$G'$ afforded by~$V$ according to Remark~\ref{VAsGPrimeModule}. For 
simplicity, identify~$G'$ with~$\PGU_3( q )$ and $\alpha$ with~$gt$. 
Then~$\chi'$ is the Steinberg character of~$\PGU_3(q)$, as the latter is the
only real extension of the Steinberg character of~$G$. Observe that 
$C_{G'}( y )$ is isomorphic to $\GU_2(q)$ for every non-trivial 
$y \in \langle gt \rangle$. Hence $\chi'( y ) = - q$ for all such~$y$; see
\cite[Theorem~$6.5.9$]{C2}. As $\chi'(1) = q^3$, this easily implies that 
$\Res^{G'}_{\langle gt \rangle}( \chi' )$ contains each of the two
real irreducible characters of $\langle gt \rangle$ as constituents.
We conclude from Lemma~\ref{MainCriterion} that $(G,V,n)$ has the 
$E1$-property.

We now consider the case of $G = G_2( q )$ with~$q$ a power of~$3$. Here, 
$\Aut(G) = \Inn(G) \rtimes \Gamma_G\Phi_G$, where $\Gamma_G\Phi_G$ is cyclic of 
order~$2f$; see \cite[Theorem~$2.5.12$(a),(e)]{GLS}. We may assume that 
$\nu \in \Gamma_G\Phi_G$. In particular,~$\nu$ is a field or a graph-field 
automorphism of~$G$ in the notation of \cite[Definition~$2.5.13$]{GLS}. Thus
$C_G( \nu ) \cong G_2( 3^{f'} )$ for some $f' \mid f$ or 
$C_G( \nu ) \cong {^2\!G}_2( 3^{2m+1} )$ for some positive integer~$m$. 
In particular,~$C_G( \nu )$ contains an involution. If~$\nu$ 
has even order, let $g := 1$, and if $\nu$ has odd order, let~$g$ denote an 
involution in~$C_G( \nu )$. Put $\alpha := \ad_g \circ \nu$. Then~$|\alpha|$ is 
even, and if~$p$ is a prime dividing~$|\alpha|$, then 
$\alpha_{(p)} \in \Gamma_G\Phi_G$, or $|g| = 2 = p$ and $\alpha_{(p)} = \ad_g$. 
If $\alpha_{(p)} \in \Gamma_G\Phi_G$, then 
$|C_G( \alpha_{(p)} )| \leq |G_2( q_0 )|$ with $q_0^p = q$ or
$|C_G( \alpha_{(p)} )| \leq |{^2\!G}_2( q )|$; see 
\cite[Proposition~$4.9.1$(a)]{GLS} and \cite[Table~$8.42$]{BHRD}. In this case, 
$|C_G( \alpha_{(p)} )| \leq q^7$. If $g \neq 1$, then  
$|C_G( \alpha_{(2)} )| \leq |\SL_2(q)|^2 \leq q^6$; see~\cite{LL}. Hence 
$$(|\alpha| - 1)|C_G( \alpha_{(p)} )|^{1/2} < q^6 = \dim(V)$$
for all primes~$p$ dividing~$|\alpha|$. Lemma~\ref{LargeDegrees} implies our 
claim.

Finally, we deal with the case $G = {^2G}_2( q )$ with $q = 3^{2m+1}$ for some 
$m \geq 1$.
Here, $\Aut(G)$ is a split extension of~$G$ with the group~$\Phi_G$ of field 
automorphisms, the latter being cyclic of order $2m+1$; see 
\cite[Theorem~$2.5.12$]{GLS}. We may thus assume that $\nu \in \Phi_G$, so 
that~$|\nu|$ is odd. Then $C_G( \nu )$ is isomorphic to ${^2G}_2( q_0 )$ for 
some root~$q_0$ of~$q$; see \cite[Proposition~$4.9.1$(a)]{GLS} and 
\cite[Table~$8.43$]{BHRD}. In particular,~$\nu$ centralizes some involution 
$g \in G$. Put 
$\alpha := \ad_g \circ \nu \in \Phi_G$.  Then $\alpha$ has even order 
dividing~$2(2m+1)$. Let~$p$ be an odd prime dividing~$|\alpha|$. Then
$C_G( \alpha_{(p)} ) \cong {^2G}_2( q_0 )$ with $q = q_0^{p}$. Also,  
$|C_G( \alpha_{(2)} )| = q(q^2-1)$; see \cite[p.~$62, 63$]{WardRee}. Hence
$|C_G( \alpha_{(2)} )| = q_0^p(q_0^{2p} - 1) \geq q_0^3(q_0^6-1) \geq 
q_0^3(q_0-1)(q_0^3+1) = |C_G( \alpha_{(p)} )|$. Now $\dim(V) = q^3 >
q[q(q^2-1)]^{1/2} > (4m + 1)[q(q^2-1)]^{1/2} \geq 
(|\alpha| - 1)|C_G( \alpha_{(p)} )|^{1/2}$
for all primes~${p}$ dividing~$|\alpha|$. The claim follows from 
Lemma~\ref{LargeDegrees}.
\end{proof}

\subsection{Reductions} We now work towards the main reductions in 
the present case.

\begin{lem}
\label{OddCharacteristicReductionI}
Let $P \leq G$ be a parabolic
subgroup of~$G$ with Levi decomposition $P = UL$, where~$U$ denotes the
unipotent radical of~$P$. Let~$S$ be an irreducible constituent of
$\Res^G_P( V )$ of odd dimension. Then~$U$ acts trivially on~$S$, i.e.\ $S$ is
a constituent of ${^*\!R}^G_L( V )$.
\end{lem}
\begin{proof}
As~$S$ is absolutely irreducible by Lemma~\ref{AbsolutelyIrreducible},  the
claim follows from Lemma~\ref{OddDegreeClifford}.
\end{proof}

\begin{prp}
\label{PrincipalSeriesReduction}
Let $B = UT$ be the Borel subgroup of~$G$. Then there is an irreducible
$\mathbb{R}T$-module~$S$ of dimension~$1$ such that~$V$ occurs in $R_T^G( S )$
with odd multiplicity. If~$\lambda$ denotes the character of~$S$, then 
$\lambda^2 = 1_T$ in the character group of~$T$.
\end{prp}
\begin{proof}
Let~$V_1$ denote a homogeneous component of $\Res^G_B( V )$ of odd dimension. 
Then~$U$ acts trivially on~$V_1$ by Lemma~\ref{OddCharacteristicReductionI}, and 
thus~$V_1$ is a homogeneous component of ${^*\!R}^G_T( V )$.
Let $S \leq V_1$ denote an irreducible $\mathbb{R}B$-submodule. Then~$S$ is also
irreducible as $\mathbb{R}T$-module. 
By adjointness of Harish-Chandra induction and Harish-Chan\-dra restriction,~$V$
is a constituent of $R_T^G(S) = \Ind_B^G( S )$ with odd multiplicity.

As $\dim_{\mathbb{R}}( S )$ is odd,~$S$ is an absolutely irreducible 
$\mathbb{R}T$-module by Lemma~\ref{AbsolutelyIrreducible}, so that 
$\dim_{\mathbb{R}}( S ) = \lambda(1) = 1$. The fact that~$\lambda$ is real 
implies that $\lambda^2 = 1_T$. This completes the proof.
\end{proof}

If~$\lambda = 1_T$ in the notation of
Proposition~\ref{PrincipalSeriesReduction}, then~$V$ is a principal series
$\mathbb{R}G$-module. In this case we get a further reduction.

\begin{lem}
\label{LeviSubgroupsReduction}
Suppose that~$V$ is a principal series $\mathbb{R}G$-module.
Let~$P$ denote a parabolic subgroup of~$G$ with Levi decomposition $P = UL$.
Suppose that there is $x \in G$ such that $\nu(U) = x^{-1}Ux$ and
$\nu(L) = x^{-1}Lx$. Suppose also that the automorphisms of~$L$ fix the
principal series characters of~$L$ and that~$L$ has the $E1$-property.

Suppose that the multiplicity of~$V$ as a direct summand
in~$R_L^G( \mathbb{R} )$ is even (including multiplicity~$0$),
where~$\mathbb{R}$ is the trivial $\mathbb{R}L$-module. Then $(G,V,n)$ has
the $E1$-property.
\end{lem}
\begin{proof}
By replacing~$n$ with $\rho(x)n$, we may assume that~$\nu$ fixes~$U$ and~$L$.

Let~$V_1$ denote a homogeneous component of $\Res_P^G( V )$ of odd dimension.
Then~$U$ acts trivially on~$V_1$ by Lemma~\ref{OddCharacteristicReductionI}.
Let $S \leq V_1$ be an irreducible $\mathbb{R}P$-submodule of~$V_1$. Then~$S$ is
also irreducible as~$\mathbb{R}L$-module.
By adjointness of Harish-Chandra induction and Harish-Chan\-dra restriction,~$V$
is a constituent of $R_L^G(S) = \Ind_P^G( S )$ with odd multiplicity.
Let~$\chi$ and~$\psi$ denote the characters of~$V$, respectively~$S$, the latter
viewed as an $\mathbb{R}L$-module. Then~$\psi$ is real and
$\langle R_L^G(\psi),\chi \rangle$ is odd.

By hypothesis,~$\psi$ is not the trivial character of~$L$. As~$\chi$ is a
principal series character by hypothesis, so is~$\psi$ by the remarks in
Subsection~\ref{HarishChandraTheory}. In particular,~$\psi$ is $\nu$-invariant
by assumption. The assertion follows from Lemma~\ref{MainHCLemma}.
\end{proof}

\begin{cor}
\label{PSCor}
Suppose that~$V$ is a principal series module which corresponds to an 
irreducible character of $W(\overline{G}^{\sigma})$ of even degree. Then
$(G,V,n)$ has the $E1$-property.
\end{cor}
\begin{proof}
If~$V$ corresponds to~$\vartheta \in \Irr( W(\overline{G}^{\sigma} ) )$, then
the multiplicity of~$V$ as a direct summand in $R_T^G( \mathbb{R} )$ 
equals~$\vartheta(1)$; see Subsection~\ref{HarishChandraTheory}. The claim
follows from Lemma~\ref{LeviSubgroupsReduction} applied to $P = B$ and $L = T$.
\end{proof}

\subsection{Non-principal series representations}
\label{NonPrincipalSeries}

Here, we consider the case that~$V$ is not in the principal series. In the
notation of Proposition~{\rm \ref{PrincipalSeriesReduction}}, which we keep
throughout this subsection, this means that $\lambda \neq 1_T$. Recall that
$\Aut(G) = \Inndiag(G) \rtimes (\Gamma_G\Phi_G)$, with 
$\Inndiag(G) = \{ \ad_h \mid h \in \overline{G}^{\sigma} \}$; see 
Subsection~\ref{AutomorphismsI}. Thus $\nu = \ad_h \circ \mu$ for some 
$h \in \overline{G}^{\sigma}$ and some $\mu \in \Gamma_G\Phi_G$.
Since $\overline{G}^{\sigma} = G\overline{T}^{\sigma}$ (see 
Subsection~\ref{BNPair}), there is $g \in G$ such that
$t := gh \in \overline{T}^{\sigma}$. By replacing~$n$ with~$\rho(g)n$, we may 
and will thus assume that $\nu = \ad_t \circ \mu$. In particular,~$\ad_t$
centralizes~$T$, and~$B$ and~$T$ are $\nu$-invariant.

\begin{lem}
\label{PoweringInvariance}
Let $\alpha \in \Aut(G)$ fix~$T$ and act on $T$ by powering its elements
(i.e.\ there is an integer~$m$ such that $\alpha(s) = s^m$ for $s \in T$).
Then~$\lambda$ is $\alpha$-invariant.
\end{lem}
\begin{proof}
Notice that~$\lambda$ is uniquely determined by $\Ker( \lambda )$, as
$\lambda^2 = 1_T$. By hypothesis,~$\alpha$ fixes every subgroup of~$T$ and
thus~$\alpha$ fixes~$\lambda$.
\end{proof}

\begin{lem}
\label{NonPrincipalSeriesCorI}
The triple $(G,V,n)$ has the $E1$-property under any of the following 
conditions.

{\rm (a)} Some $N$-conjugate of~$\lambda$ is $\nu$-invariant.

{\rm (b)} The torus~$T$ is cyclic.

{\rm (c)} The automorphism $\mu$ acts on~$T$ by powering its elements.
\end{lem}
\begin{proof}
(a) Observe that~$\lambda$ is cuspidal. If $\lambda'$ is an $N$-conjugate 
of~$\lambda$, then $R_T^G( \lambda ) = R_T^G( \lambda' )$; see 
\cite[Proposition~$8.2.7$(ii)]{C2}. If, in addition,~$\lambda'$ is
$\nu$-invariant, $(G,V,n)$ has the $E1$-property by Lemma~\ref{MainHCLemma},
as~$T$ has the $E1$-property by Corollary~\ref{SolvableGroups}.

(b) Since~$T$ is cyclic, $\lambda \in \Irr(T)$ is uniquely determined by
$\lambda^2 = 1_T \neq \lambda$. Thus~$\lambda$ is $\nu$-invariant and the
claim follows from~(a). 

(c) This follows from Lemma~\ref{PoweringInvariance} and~(a).
\end{proof}

\begin{cor}
\label{NonPrincipalSeriesCorII}
Suppose that~$G$ is not one of the groups in rows~$1$,~$5$ or~$9$ of 
{\rm Table~\ref{TableOfGroups}}. Then~$(G,V,n)$ has the $E1$-property.
\end{cor}
\begin{proof}
If $G = {^2G}_2( 3^{2m+1} )$ for some $m \geq 1$, then~$T$ is cyclic. In the 
other cases, $\Aut(G) = \Inndiag(G) \rtimes \Phi_G$, and the elements 
of~$\Phi_G$ act on~$T$ by powering its elements. Our assertion thus follows 
from Lemma~\ref{NonPrincipalSeriesCorI}(b),(c).
\end{proof}

In the proofs of 
Propositions~\ref{NonPrincipalSeriesCorIII}--\ref{NonPrincipalSeriesCorV} below,
we let $\kappa \colon \mathbb{F}_q^* \rightarrow \mathbb{C}^*$ denote the unique
irreducible character of order~$2$.

\begin{prp}
\label{NonPrincipalSeriesCorIII}
Suppose that $G = \PSL_d(q)$ for some $d \geq 2$ and some prime power~$q$.
Then~$(G,V,n)$ has the $E1$-property.
\end{prp}
\begin{proof}
Let us work with $\tilde{G} := \SL_d(q)$, and view the characters of~$G$ as
characters of~$\tilde{G}$ via inflation. Let~$\tilde{T}$ and~$\tilde{N}$ denote 
the inverse images of~$T$, respectively~$N$, under the canonical epimorphism
$\tilde{G} \rightarrow G$. Then~$\tilde{T}$ is the standard
torus of~$\tilde{G}$, consisting of the diagonal matrices of determinant~$1$.
For $1 \leq i \leq d$, let
$\lambda_i\colon \tilde{T} \rightarrow \mathbb{C}^*, \diag(t_1, \ldots , t_d)
\mapsto \kappa( t_i )$. Then every irreducible character of~$\tilde{T}$ is a
product of some $\lambda_i$s. For $I \subset \{ 1, \ldots , d \}$ write
$\lambda_I := \prod_{i \in I} \lambda_i$. Notice that $\lambda_I = \lambda_{I'}$
with $I' = \{ 1, \ldots , d \} \setminus I$, and that 
$\lambda_{\{ 1, \ldots , d \}} = 1_{\tilde{T}}$. Notice also that for 
$I, J \subseteq \{ 1, \ldots , d \}$, the characters $\lambda_I$ and~$\lambda_J$ 
are conjugate by an element of $\tilde{N}$, if and only if $|I| = |J|$ or 
$|I| = d - |J|$. Finally, if~$\iota$ denotes the standard graph automorphism 
of~$\tilde{G}$, then ${^\iota\!\lambda}_i = \lambda_{d-i+1}$, for 
$1 \leq i \leq d$.

Let $I \leq \{ 1, \ldots , d \}$ be such that $\lambda = \lambda_I$. By the 
remarks in the previous paragraph, the $\tilde{N}$-orbit of~$\lambda$ contains a 
$\iota$-stable element, if either~$d$ is odd, or~$d$ and~$|I|$ are even. Then
the $N$-orbit of~$\lambda$ contains a $\nu$-stable element, and we are done by 
Lemma~\ref{NonPrincipalSeriesCorI}(a).

Suppose then that~$d$ is even and~$|I|$ is odd. If $d = 2$, then~$T$ is cyclic
and our claim follows from Lemma~\ref{NonPrincipalSeriesCorI}(b). Suppose then 
that $d \geq 4$. Let~$\tilde{L}$ denote the $\iota$-invariant Levi subgroup of 
type~$A_1$ corresponding to the central node of the Dynkin diagram 
of~$\tilde{G}$; see Figure~\ref{DynkinDiagramA}.  Assume, without loss of 
generality, that $|I \cap \{ d/2, d/2 + 1 \}| = 1$, and that 
$I \setminus \{ d/2, d/2 + 1 \}$ is invariant under reversing the elements. 
Then the orbit of~$\lambda$ under $\langle \iota \rangle$ equals the orbit 
of~$\lambda$ under $N_{\tilde{L}}( \tilde{T} )$. This shows that 
$\psi := R_{\tilde{T}}^{\tilde{L}}( \lambda )$ is irreducible and fixed 
by~$\iota$. We may thus assume that~$\psi$ is $\nu$-stable. As the 
character~$\chi$ of~$V$ occurs in $R_{\tilde{L}}^{\tilde{G}}( \psi ) = 
R_{\tilde{L}}^{\tilde{G}}( R_{\tilde{T}}^{\tilde{L}}( \lambda ) ) = 
R_{\tilde{T}}^{\tilde{G}}( \lambda )$ with odd multiplicity, we are done
with Lemma~\ref{MainHCLemma}.
\end{proof}

\begin{prp}
\label{NonPrincipalSeriesCorIV}
If $G = E_6(q)$, then~$(G,V,n)$ has the $E1$-property.
\end{prp}
\begin{proof}
Write $\mu = \iota \circ \mu'$ for some $\iota \in \Gamma_G$ and some 
$\mu' \in \Phi_G$. As~$\mu'$ acts on~$T$ by powering its elements, it suffices 
to show that some $N$-conjugate of~$\lambda$ contains a $\iota$-stable element; 
see Lemma~\ref{PoweringInvariance} and 
Lemma~\ref{NonPrincipalSeriesCorI}(a). This is trivial if $\iota = \id$.
We may thus assume that~$\iota$ equals the non-trivial graph automorphism 
of~$G$. By its very definition,~$\iota$ extends to the graph automorphism 
of~$\overline{G}$ defined in~\cite[Theorem~$1.15.2$]{GLS}; see 
\cite[Definition~$2.5.10$(b)]{GLS}. This extension, as well as its restriction 
to~$\overline{G}^{\sigma}$, are also denoted by~$\iota$.

Since $[\overline{G}^{\sigma}\colon\!G] = [\overline{T}^{\sigma}\colon\!T]$ is 
odd, restriction of characters yields a bijection between 
$\{ \tau \in \Irr( \overline{T}^{\sigma} ) \mid \tau^2 = 1_{\overline{T}^{\sigma}} \}$ 
and $\{ \tau \in \Irr( T ) \mid \tau^2 = 1_{T} \}$. 
The mapping $\psi \mapsto \kappa \circ \psi|_{\overline{T}^{\sigma}}$ 
for $\psi \in X( \overline{T } )$, yields a $\langle \iota \rangle$-equivariant
isomorphism 
$$X(\overline{T})/2X(\overline{T})  \rightarrow
\{ \tau \in \Irr( \overline{T}^{\sigma} ) \mid \tau^2 = 1_{\overline{T}^{\sigma}} \}.$$
As~$\overline{G}$ is of adjoint type,~$X(\overline{T})$ has a basis 
consisting of the set~$\Pi$ of simple roots. Using this, it is easy to check 
with Chevie~\cite{chevie} that the Weyl group 
$N_{\overline{G}^{\sigma}}( \overline{T} )/\overline{T}^{\sigma} = N/T$ 
of~$\overline{G}^{\sigma}$ has exactly three orbits 
on~$X(\overline{T})/2X(\overline{T})$, each of which contains a $\iota$-stable 
element. 

Hence some $N$-conjugate of~$\lambda$ is $\nu$-stable and we are done by
Lemma \ref{NonPrincipalSeriesCorI}(a).
\end{proof}

\begin{prp}
\label{NonPrincipalSeriesCorV}
If $G = \P\Omega^+_{2d}(q)$ for some $d \geq 4$, then~$(G,V,n)$ has the 
$E1$-property.
\end{prp}
\begin{proof}
Write $\mu = \iota \circ \mu'$ for some $\iota \in \Gamma_G$ and some 
$\mu' \in \Phi_G$. As in the proof of Proposition~\ref{NonPrincipalSeriesCorIV}, 
it suffices to show that some $N$-conjugate of~$\lambda$ contains a 
$\iota$-stable element, and we may thus assume that~$\iota$ is non-trivial.

We begin with the case $|\iota| = 2$. Then~$\iota$ is $\Gamma_G$-conjugate to
the standard graph automorphism of~$G$, and we will assume that~$\iota$ is 
equal to the latter. We claim that~$\iota$ stabilizes~$\lambda$.
To prove this, consider the group $\tilde{G} := \SO^+_{2d}(q)$, which 
contains~$G$ as a composition factor. Indeed, the commutator subgroup 
of~$\tilde{G}$ equals $\Omega^+_{2d}(q)$, and~$G$ is the quotient of the latter 
by its center; see \cite[Section~$11$]{Taylor}.

To realize~$\tilde{G}$ as a matrix group, equip the standard vector space
$\mathbb{F}_q^{2d}$ with a non-degenerate symmetric bilinear form, and 
choose a basis $e_1, \ldots , e_d, e_d', \ldots , e_1'$ such that $
(e_i, e_i')$ is a hyperbolic pair for all $1 \leq i \leq d$. Then, with respect
to this basis, $\tilde{G} = \{ x \in \SL_{2d}(q) \mid x^tJx = J \}$, 
where~$J$ is the matrix with $1$'s along the antidiagonal, and $0$'s, elsewhere.
The standard torus~$\tilde{T}$ of~$\tilde{G}$ is given by
$$\tilde{T} = 
\{ \diag( \zeta_1, \ldots, \zeta_d, \zeta_d^{-1}, \ldots , \zeta_1^{-1} ) \mid 
\zeta_1, \ldots, \zeta_d \in \mathbb{F}_q^* \}.$$
Then~$T$ is a quotient of $\tilde{T} \cap \Omega^+_{2d}(q)$, and we may 
view~$\lambda$ as an irreducible character of $\tilde{T} \cap \Omega^+_{2d}(q)$ 
via inflation. Now $\iota$ is induced by the standard graph automorphism 
of~$\tilde{G}$ of order~$2$. The latter, denoted by $\tilde{\iota}$, is realized 
by conjugating~$\tilde{G}$ with the element of $\GL_{2d}(q)$ which swaps the 
basis elements $e_d$ and $e_d'$, and fixes all the others.

Observe that~$\lambda$, viewed as an irreducible character of 
$\tilde{T} \cap \Omega^+_{2d}(q)$, extends to some
$\tilde{\lambda} \in \Irr( \tilde{T} )$ satisfying 
$\tilde{\lambda}^2 = 1_{\tilde{T}}$, and it suffices to show 
that~$\tilde{\iota}$ fixes~$\tilde{\lambda}$. 
Let $\tilde{\lambda}_i \colon \tilde{T} \rightarrow \mathbb{C}^*$,
$\diag( \zeta_1, \ldots, \zeta_d, \zeta_d^{-1}, \ldots , \zeta_1^{-1} ) \mapsto
\kappa( \zeta_i )$ for $1 \leq i \leq d$. Then $\tilde{\lambda} \in 
\langle \tilde{\lambda}_1, \ldots , \tilde{\lambda}_d \rangle$. 
Clearly,~$\tilde{\lambda}_i$ is $\tilde{\iota}$-invariant for all 
$1 \leq i \leq d$, which proves our claim. This gives our assertion in case
$|\iota| = 2$.

We are left with the case that $d = 4$ and $|\iota| = 3$. Recall that
$\overline{T}^{\sigma}/T \cong \overline{G}^{\sigma}/G$. It thus follows
from \cite[Table~$6.1.2$]{GLS} and \cite[Theorem~$2.5.20$]{GeMa} that
$\overline{T}^{\sigma}/T$ is elementary abelian of order~$4$. 
In particular,~$\lambda$ extends to 
four distinct characters $\bar{\lambda} \in \Irr( \overline{T}^{\sigma} )$ with 
$\bar{\lambda}^2 = 1_{\overline{T}^{\sigma}}$. We now argue as in the proof of 
Proposition~\ref{NonPrincipalSeriesCorIV}. Using Chevie, we find that the Weyl 
group $N_{\overline{G}^{\sigma}}( \overline{T} )/\overline{T}^{\sigma}$
of~$\overline{G}^{\sigma}$ has exactly five orbits 
on~$X(\overline{T})/2X(\overline{T})$, four of length~$1$ and one of 
length~$12$. 
The latter orbit contains the image of a root, as well as a $\iota$-stable 
element, namely the image of the highest root. If $\bar{\lambda}' \in 
\Irr( \overline{T}^{\sigma} )$ corresponds to an image of a root, then 
$\Res_T^{\overline{T}^{\sigma}}( \bar{\lambda}' )$ is non-trivial. Thus this 
orbit of length~$12$ accounts for the three irreducible characters of~$T$ 
of order~$2$. In particular, some $N$-conjugate of~$\lambda$ is $\iota$-stable.
This completes the proof.
\end{proof}

\subsection{Principal series representations}
\label{PSExceptional}
Here, we consider the case that $\lambda = 1_T$, i.e.\ that~$V$ is a principal 
series module. As explained in the introduction to 
Subsection~\ref{NonPrincipalSeries}, we may assume that 
$\nu = \ad_t \circ \mu$ for some $t \in \overline{T}^{\sigma}$ and some 
$\mu \in \Gamma_G\Phi_G$. In order to apply Lemma~\ref{LeviSubgroupsReduction},
we choose~$P$ and~$L$ as a standard parabolic subgroup and a standard Levi 
subgroup, respectively. Then~$P$ and~$L$ are $\ad_t$-invariant. With a suitable
choice, we can also achieve that~$P$ and~$L$ are~$\mu$-invariant. Working 
inductively, we may assume that~$L$ has the $E1$-property. If the remaining
hypothesis of this lemma is satisfied, only the constituents of 
$R_L^G( \mathbb{R} )$ occurring with odd multiplicity have to be considered. 
If~$L$ is a large Levi subgroup, the number of such constituents is small, which 
restricts the possible $\mathbb{R}G$-modules~$V$ to be investigated.

We begin with the exceptional groups.

\begin{prp}
\label{ExceptionalGroupsOddCharacteristic}
Let~$G$ be an exceptional group of Lie type, such that every proper subgroup 
of~$G$ has the $E1$-property. Then~$(G,V,n)$ has the $E1$-property.
\end{prp}
\begin{proof}
By Table~\ref{TableOfGroups}, the group~$G$ is one of $G_2(q)$, 
$F_4(q)$, $E_6(q)$, $E_7(q)$, $E_8(q)$, ${^2\!E}_6(q)$ or ${^3\!D}_4(q)$ 
with~$q$ odd, or a Ree group $G = {^2G}_2(3^{2m+1})$ for some positive 
inter~$m$. 
	
The only non-trivial principal series module of a Ree group is its 
Steinberg module, so that the claim for these groups follows from 
Lemma~\ref{SmallRankOddCharacteristic}. 
Thus, let us assume that~$G$ is not one of the groups ${^2G}_2(3^{2m+1})$ in 
the following.

The table below specifies, for each~$G$, two or three standard Levi subgroups 
of~$G$ as follows. The second column gives the Weyl 
group~$W(\overline{G})^{\sigma}$ of~$G$, denoted by its Dynkin type. The third 
column lists subgraphs of the Dynkin diagram of $W(\overline{G})^{\sigma}$; 
these determine parabolic subgroups of~$W(\overline{G})^{\sigma}$ and also 
standard parabolic subgroups and standard Levi subgroups of~$G$. In case when 
$W(\overline{G})^{\sigma}$ is of type~$G_2$, the two subgraphs~$A_1$ 
and~$\tilde{A}_1$ correspond to the long and to the short root of the 
fundamental system of this type, respectively.
	
\setlength{\extrarowheight}{0.5ex}
$$
\begin{array}{ccccc} \hline\hline
        G & W(\overline{G})^{\sigma} & L, L_1 & \text{Const.} & \text{Mult.} \\ \hline\hline
        G_2(q), {^3\!D}_4(q) & G_2 & A_1 & {\phi_{1,3}}'' & 1 \\ 
        & & \tilde{A}_1 & & 0 \\ \hline
        F_4(q), {^2\!E}_6(q) & F_4 & B_3 & \phi_{9,2} & 1 \\ 
        & & A_1 & & 6 \\ \hline
        E_6(q) & E_6 & A_5 & \phi_{15,4} & 1 \\
        & & A_1 & & 10 \\ \hline
        E_7(q) & E_7 & E_6 & \phi_{7,1}, \phi_{21,3}, \phi_{27,2} & 1,1,1 \\ 
        & & D_5 \times A_1 & & 1, 1, 2 \\ 
        & & A_1 & & 6, 16, 21 \\ \hline
        E_8(q) & E_8 & E_7 & \phi_{35,2} & 1 \\ 
        & & A_1 & & 28 \\ \hline\hline
\end{array}
$$

\medskip

Let~$P$ denote the standard parabolic subgroup of~$G$ specified in the first 
row of this table corresponding to~$G$, and let~$P_1$ denote the one specified in 
the other row, respectively, in case of $G = E_7(q)$, in one of
the two other rows. We have Levi decompositions $P = UL$ and $P_1 = U_1L_1$ with 
the standard Levi subgroups~$L$ and~$L_1$, and the unipotent radicals~$U$ 
and~$U_1$ of~$P$ and~$P_1$, respectively. The column of the table headed 
``Const.'' gives all non-trivial constituents of $R_{L}^G( \mathbb{R} )$ of odd 
dimension, denoted by their labels as in~\cite[Subsection~$13.9$]{C1}. The last 
column contains the multiplicities of these constituents in 
$R_{L}^G( \mathbb{R} )$ and in $R_{L_1}^G( \mathbb{R} )$, respectively. By the 
Howlett-Lehrer comparison theorem \cite[Theorem 5.9]{HowLeh2}, these 
multiplicities can be computed by inducing the trivial characters of the 
corresponding parabolic subgroups of~$W(\overline{G})^{\sigma}$ 
to~$W(\overline{G})^{\sigma}$; see also Subsection~\ref{HarishChandraTheory}. 
These computations are easily performed with Chevie~\cite{chevie}.

Recall that $\nu = \ad_t \circ \mu$ for some $t \in \overline{T}^{\sigma}$ and
some $\mu \in \Gamma_G\Phi_G$. If $G = G_2(q)$ with $q = 3^f$, 
then~$\Gamma_G\Phi_G$ is cyclic of order $2f$; see 
\cite[Theorem~$2.5.12$(d),(e)]{GLS}. In this case, a generator~$\psi$ 
of~$\Gamma_G\Phi_G$ swaps the two standard Levi subgroups of types~$A_1$
and $\tilde{A}_1$. In all the other cases,~$\mu$, and hence~$\nu$,
stabilizes the groups~$P$,~$L$,~$P_1$ and~$L_1$, and thus also~$U$ and~$U_1$. 
Moreover, every automorphism of~$L$ and~$L_1$ fixes their principal series
characters; see \cite[Theorem~$2.5$(a)]{MalleExt}. By assumption,~$L$ and~$L_1$ 
have the $E1$-property.  If~$V$ is a constituent with even multiplicity 
(including multiplicity~$0$) in one of $R_{L}^G( \mathbb{R} )$ or 
$R_{L_1}^G( \mathbb{R} )$, then $(G,V,n)$ has the $E1$-property by 
Lemma~\ref{LeviSubgroupsReduction}. 

Thus $(G,V,n)$ has the $E1$-property, except, possibly, if $G = G_2(q)$ with 
$q = 3^f$ and~$\mu$ is an odd power of~$\psi$. In this case, the non-trivial 
principal series characters of~$G$ of odd dimension are ${\phi_{1,3}}'$, 
${\phi_{1,3}}''$ and $\phi_{1,6}$, the latter being the Steinberg character. 
As~$\nu$ swaps the two representations ${\phi_{1,3}}'$ and ${\phi_{1,3}}''$ by 
\cite[Theorem~$2.5$(d)]{MalleExt}, we are left with the case that~$V$ is the
Steinberg module of~$G$. This case is settled
in Lemma~\ref{SmallRankOddCharacteristic}.
\end{proof}

We finally deal with the classical groups. 

\begin{prp}
\label{OddCharacteristicClassicalGroups}
Let~$G$ be a classical group, such that every proper subgroup of~$G$ has the 
$E1$-property. Then~$(G,V,n)$ has the $E1$-property.
\end{prp}
\begin{proof}
By Table~\ref{TableOfGroups}, the group~$G$ is one of 
$\PSL_d( q )$, $d \geq 2$, $\PSU_d(q)$, $d \geq 3$, $\P\Omega_d(q)$, 
$d \geq 5$ odd, $\PSp_d(q)$, $d \geq 6$ even, or
$\P\Omega^{\pm}_d( q )$, $d \geq 8$ even.
 
By hypothesis,~$V$ is a principal series $\mathbb{R}G$-module. Let~$\chi$ denote
the character of~$V$.

We write $W := W( \overline{G} )^{\sigma}$ for the Weyl group of~$G$. Then~$W$
is a Coxeter group of type~$A_{d-1}$, of type~$B_d$ or~$D_d$. In the respective
cases, the irreducible characters of~$W$ are labelled by partitions of~$d$, by
bipartitions of~$d$, and by signed, unordered bipartitions of~$d$; see, e.g.\
\cite[Section~$13.2$]{C2}. Following the
usage announced in Subsection~\ref{HarishChandraTheory}, we will 
write~$\chi^{\pi}$ for the character of the principal series module of~$G$ 
corresponding, via Harish-Chandra theory, to the irreducible character of~$W$ 
labelled by~$\pi$; the latter will be denoted by~$\zeta^{\pi}$. Recall that the 
computation of Harish-Chandra induced trivial modules is reduced to the 
induction of the trivial character from parabolic subgroups of~$W$. This is 
usually done with the branching rules. 

Let us begin with the case $G = \PSL_d( q )$, $d \geq 2$. If $d < 4$, the only 
non-trivial principal series module of~$G$ of odd dimension is the Steinberg 
module. For this, our assertion has been settled in 
Lemma~\ref{SmallRankOddCharacteristic}. So let us assume that $d \geq 4$ in the 
following. Here,~$W$ is a Coxeter group of type~$A_{d-1}$, i.e.\ $W \cong S_d$, 
the symmetric group on~$d$ letters. First, consider the standard 
Levi subgroup~$L_I$ for $I \subseteq \Pi$ of type $A_{d-3}$, invariant under the 
graph automorphism; i.e.\ $I$ is obtained from deleting the first and the last 
node of the Dynkin diagram of type $A_{d-1}$; see Figure~\ref{DynkinDiagramA}. 
Then~$P_I$ and~$L_I$ are $\nu$-invariant. By the branching rule for~$S_d$, the 
non-trivial constituents of $R^G_{L_I}( \mathbb{R} )$ have 
characters $\chi^{\pi}$ with $\pi \in \{ (d-1,1), (d-2,2), (d-2,1^2) \}$, where
$\chi^{(d-1,1)}$ occurs twice. By Lemma~\ref{LeviSubgroupsReduction}, we may 
assume that $\chi = \chi^{\pi}$ with~$\pi$ one of~$(d-2,2)$ or~$(d-2,1^2)$.
	
Now consider the standard Levi subgroup $L_{I'}$ of type $A_1 \times A_1$ 
corresponding to the two outer nodes of the Dynkin diagram of type $A_{d-1}$.
Again, $P_{I'}$ and $L_{I'}$ are $\nu$-invariant. Let~$\psi$ denote the 
Steinberg character of~$L_{I'}$.
Clearly,~$\psi$ is $\nu$-invariant, real and of odd degree. By 
Lemma~\ref{MainHCLemma}, it suffices to show that 
$\langle R_{L_{I'}}^G( \psi ), \chi \rangle$ is odd. We claim that, in fact,
$\langle R_{L_{I'}}^G( \psi ), \chi \rangle = 1$. To prove this
claim, observe that~$\psi$ corresponds to the sign-character~$\xi$ of the 
parabolic subgroup $W( \overline{G} )^{\sigma}_{I'} \cong S_2 \times S_2$ 
of $W \cong S_d$. An application of the branching rule shows 
that
$$\langle \Ind_{S_2 \times S_2}^{S_d}( \xi ), \zeta^{\pi} \rangle = 1$$
for $\pi \in \{ (d-2,2), (d-2,1^2) \}$. The Howlett-Lehrer comparison theorem 
\cite[Theorem 5.9]{HowLeh2} proves our claim. Thus~$(G,V,n)$ has the
$E1$-property if $G = \PSL_d( q )$.

Next, assume that~$G$ is one of the groups $\PSU_{2d}(q)$, $\PSU_{2d+1}(q)$,
$\PSp_{2d}(q)$, $\P\Omega_{2d+1}(q)$ or $\P\Omega^-_{2(d+1)}(q)$. In this case,
all standard parabolic subgroups and standard Levi subgroups are $\mu$-invariant 
and hence $\nu$-invariant. The Weyl~$W$ of~$G$ is of type $B_d$; see, e.g.\ 
\cite[Proposition~$2.3.2$]{GLS}. By
Lemma~\ref{SmallRankOddCharacteristic} we may assume that $d \geq 2$ if 
$G = \PSU_{2d+1}(q)$, as the only non-trivial principal series module of 
$\PSU_3(q)$ is the Steinberg module, which has been dealt with in 
Lemma~\ref{SmallRankOddCharacteristic}. Hence $d \geq 2$ in all cases. Applying 
Lemma~\ref{LeviSubgroupsReduction} to the standard parabolic subgroup of type 
$B_{d-1}$, we may assume that
$\chi = \chi^{\pi}$ with $\pi \in \{  (d-1,1), - ), ( (d-1), (1) ) \}$.
Using Corollary~\ref{PSCor} and, once more,
Lemma~\ref{LeviSubgroupsReduction}, the following table, which is proved by the 
branching rules, establishes our claim. In this table, the subgroup~$S_2$ of~$W$ 
corresponds to the node of valency one at the end of the Dynkin diagram of~$B_d$ 
if $d > 2$, and if $d = 2$, to one of the two nodes of the Dynkin diagram 
of~$B_2$.

\setlength{\extrarowheight}{0.7ex}
$$
\begin{array}{ccc} \hline\hline
\pi & \zeta^{\pi}(1) & 
\langle \zeta^{\pi}, \Ind_{S_2}^{W}( 1_{S_2} ) \rangle \\ \hline\hline
( (d-1,1), - ) & d - 1 & d - 2 \\
( (d-1), (1) ) & d & d - 1 \\ \hline\hline
\end{array}
$$

\medskip
We are left with the case $G = \P\Omega^+_{2d}(q)$ with $d \geq 4$. In this case,
the Weyl group~$W$ of~$G$ is a Coxeter group of type $D_d$. We first deal with 
the case $d = 4$, where there is an exceptional graph automorphism of order~$3$ 
of the Dynkin diagram of~$D_4$; see Figure~\ref{DynkinDiagramA}. Let 
$I' \subseteq \Pi$ be of type $A_1$, invariant under all graph automorphisms of 
the Dynkin diagram, i.e.\ $I'$ corresponds to node~$3$ in 
Figure~\ref{DynkinDiagramA}. Then~$P_{I'}$ and~$L_{I'}$ are $\nu$-invariant.
Applying Corollary~\ref{PSCor} and Lemma~\ref{LeviSubgroupsReduction}, we may
assume that $\chi = \chi^{\pi}$ with $\pi$ one of 
$\{ (1^2), (1^2) \}^+$, $\{ (1^2), (1^2) \}^-$ or $\{ -, (2,1^2) \}$. 
By \cite[Theorem~$2.5$(b)]{MalleExt}, the elements of order~$3$ in~$\Gamma_G$
permute the three characters with labels 
$\{ (1^2), (1^2) \}^+$, $\{ (1^2), (1^2) \}^-$ and $\{ -, (2,1^2) \}$ 
transitively. We have $\nu = \ad_t \circ \iota \circ \varphi'$ with $\iota \in
\Gamma_G$ and $\varphi' \in \Phi_G$. By, \cite[Theorem~$2.5$]{MalleExt}, the 
unipotent characters of~$G$ are fixed by~$\ad_t$ and~$\varphi'$. Since~$\chi$ is 
$\nu$-invariant and~$\Gamma_G$ is isomorphic to~$S_3$, we conclude that 
$|\iota| \leq 2$. This implies that there is a $\iota$-invariant $3$-element 
subset $I \subseteq \Pi$ containing the central node of the Dynkin diagram.
Then $P_{I}$ and $L_{I}$ are $\nu$-invariant. The non-trivial constituents of 
$R_{L_I}^G( \mathbb{R} )$ are labeled by the unordered bipartitions 
$\{(1), (3)\}$ and $\{ - , (3,1)\}$. Hence~$V$ does not occur in 
$R_{L_{I}}^G( \mathbb{R} )$, and we conclude that~$V$ has the $E1$-property 
from Lemma~\ref{LeviSubgroupsReduction}.

Finally, assume that $d > 4$. Here, we take $I$ of type $D_{d-1}$ and $I'$ of
type~$A_1$, the latter corresponding to the leaf of the Dynkin diagram whose
removal gives~$I$. Then~$I$ and $I'$ are invariant under the symmetries of the
Dynkin diagram, and thus $P_I$, $L_I$, $P_{I'}$ and $L_{I'}$ are 
$\nu$-invariant. The non-trivial constituents of $R_{L_{I}}( \mathbb{R} )$ are
labeled by the unordered bipartitions $\{ (1), (d-1) \}$ and $\{ -, (d-1,1) \}$.
Then the above table for the Coxeter group of type~$B_d$ also works for $W$ of 
type $D_d$, and we are done by Corollary~\ref{PSCor} and
Lemma~\ref{LeviSubgroupsReduction}.
\end{proof}

\subsection{Summary} We summarize our results for the finite simple groups of
Lie type of odd characteristic.

\begin{thm}
\label{MainOdd}
Let~$G$ be a finite group of Lie type of odd characteristic. Then~$G$ is not a
minimal counterexample to {\rm Theorem~\ref{RLTheorem}}.
\end{thm}
\begin{proof}
This follows from Proposition~\ref{PrincipalSeriesReduction},
Corollary~\ref{NonPrincipalSeriesCorII}, 
Propositions~\ref{NonPrincipalSeriesCorIII}--\ref{NonPrincipalSeriesCorV}, 
Proposition~\ref{ExceptionalGroupsOddCharacteristic} 
and Proposition~\ref{OddCharacteristicClassicalGroups}.
\end{proof}

\section*{Acknowledgements}
The authors thank, in alphabetical order, Thomas Breuer, Xenia Flamm, Meinolf 
Geck, Jonas Hetz, Frank Himstedt, Frank L{\"u}beck, Jean Michel, 
Britta Sp{\"a}th, Andrzej Szczepa{\'n}ski and Jay Taylor for numerous helpful 
hints and enlightening conversations.

\end{document}